\documentclass[english]{article}
\usepackage[T1]{fontenc}
\usepackage[latin9]{inputenc}
\usepackage{color}
\usepackage{amsmath}
\usepackage{amssymb}
\usepackage{amssymb}
\usepackage{amsmath}
\usepackage{amssymb}
\usepackage{amsthm}
\makeatletter
\usepackage{fullpage}\usepackage{makeidx}\usepackage{amsfonts}\usepackage{latexsym}

\def\<{{\langle}}
\def\>{{\rangle}}
\theoremstyle{definition}
\newtheorem{lemma}{Lemma}[section]
\newtheorem{theorem}[lemma]{Theorem}\newtheorem{proposition}[lemma]{Proposition}\newtheorem{definition}[lemma]{Definition}\newtheorem{remark}[lemma]{Remark}\usepackage{times}

\usepackage{babel}

\usepackage{enumitem}

\usepackage[blocks]{authblk}% The option is for block layout
\title{Fusion Products of Twisted Modules in Permutation Orbifolds: II}
\author{Chongying Dong, Feng Xu, Nina Yu}
\date{}

\numberwithin{equation}{section}

\makeatother

\usepackage{babel}
\begin{document}
\maketitle
\begin{abstract}
Let $V$ be a simple, rational, $C_{2}$-cofinite vertex operator algebra of CFT type, and let $k$ be a positive integer. In this paper, we determine the fusion products of twisted modules for $V^{\otimes k}$ and  $G = \left\langle g \right\rangle$ generated by any permutation $g \in S_{k}$. 
\end{abstract}

\tableofcontents{}

\section{Introduction}

This paper is a continuation of \cite{DLXY}, investigating  the fusion products in permutation orbifold theory.  The permutation orbifolds study the tensor
product vertex operator algebra $V^{\otimes k}$ with  the symmetric group $S_{k}$ acting as automorphisms. Here,
$V$ is a vertex operator algebra and $k$ is a positive integer.

The exploration of permutation orbifolds for general vertex operator algebras commenced with \cite{BDM}, which uncovered a fundamental relationship between twisted modules for $V^{\otimes k}$
  under permutation automorphisms and 
$V$-modules. In particular, if $g$ is a $k$-cycle,  a canonical $g$-twisted $V^{\otimes k}$-module
structure on any $V$-module $(W,Y_{W}(\cdot,z))$ was discovered, denoted as $T_{g}\left(W\right)$. This construction established an isomorphism between categories of weak, admissible, and ordinary $V$-modules and their corresponding $g$-twisted $V^{\otimes k}$-modules. 

The permutation orbifolds have been studied extensively. 
The $C_{2}$-cofiniteness of permutation orbifolds and general cyclic
orbifolds was established in works such as  \cite{A1,A2,M1,M2}. An equivalence of two constructions (see \cite{FLM}, \cite{Le},
\cite{BDM}) of twisted modules for the permutation orbifolds of lattice
vertex operator algebras was established  in \cite{BHL}. Studies on permutation orbifolds within the framework of conformal nets, including the determination of irreducible representations and fusion rules for $k= 2$, were presented in \cite{KLX}. Research on permutation orbifolds of lattice vertex operator algebras for $k = 2$ and $k = 3$ is presented in \cite{DXY1, DXY2, DXY3}, focusing on the classification of irreducible modules and the determination of fusion rules. The $S_3$-orbifold of a rank-three Heisenberg vertex algebra  is examined in \cite{MPSh}. Permutation orbifolds of the Virasoro vertex algebra for $k = 2$ and $k = 3$ are studied in \cite{MPS1} and \cite{MPS2}. The representation theory of a permutation orbifold of a lattice vertex algebra $V_Q$ is investigated in \cite{BEKT}, with an emphasis on irreducible $V_Q^\sigma$-modules and their characters, where $Q$ is a direct sum of a prime number of copies of a positive-definite even lattice, and $\sigma$ acts as a cyclic shift.
  An explicit expression for the S-matrix of the orbifold $(V^{\otimes k})^{\langle g \rangle}$, where $g = (1, 2, \dots, k)$, was derived in \cite{DXY4} in terms of the S-matrix of $V^{\otimes k}$.
The fusion products of $V^{\otimes k}$-modules
with $g$-twisted $V^{\otimes k}$-modules for any $g\in S_{k}$ were
  thoroughly determined in \cite{DLXY}. 
  
  In this paper, we extend previous studies by investigating the fusion products of twisted $V^{\otimes k}$-modules when
  $G$ is a cyclic group.  We derive explicit formulas for these fusion products in terms of the entries of the \( S \)-matrix of \( V \). The fusion products of twisted modules for any cyclic group generated by a permutation can be reduced to the case where the permutation is a single cycle.
To give a precise description we let \text{Irr}$(V)=\{M^0=V, M^1,\cdots, M^p\}$ be all irreducible $V$-modules.  
  Let  $g=(1, 2,\cdots,k)$ and  $1 \leq s, r < k$. By \cite{BDM},   any $g^s$-twisted module has the form $T_{g^s}^{M^{i_1},\cdots,M^{i_d}}$ where $d=\text{gcd}(s,k)$ and $M^{i_1},\cdots,M^{i_d}$ are irreducible $V$-modules. Similarly, any $g^r$-twisted modules is given by $T_{g^r}^{M^{j_1},\cdots,M^{j_m}}$ where $m=\text{gcd}(r,k)$ and $M^{j_1},\cdots,M^{j_m}$ are irreducible $V$-modules. To determine the fusion products of $T_{g^s}^{M^{i_1},\cdots,M^{i_d}}$ with $T_{g^r}^{M^{j_1},\cdots,M^{j_m}}$, we first find the fusion product of $T_{g^s}^{V,\cdots,V} $ with $T_{g^r}^{V,\cdots,V}$. %Note that all $g^t$-twisted $V^{\otimes k}$-modules for $t=0,...,k-1$ form a fusion category.  We then apply the fusion product \( T_{g^s}^{V, \dots, V} \boxtimes T_{g^r}^{V, \dots, V} \) and
  We then use results on the fusion product of untwisted and twisted modules obtained in \cite{DLXY} to determine the fusion product \( T_{g^s}^{M^{i_1}, \dots, M^{i_d}} \boxtimes T_{g^r}^{M^{j_1}, \dots, M^{j_m}} \) in general.

The structure of this paper is as follows: Section 2 introduces the basic notions and results related to vertex operator algebras. Section 3 reviews fundamental concepts on fusion products and the \(S\)-matrix in permutation orbifolds. In Section 4, we first show that the fusion products of any $\sigma^s$-twisted module with any $\sigma^r$-twisted module, where $\sigma \in S_k$, can be reduced to the case that $\sigma$ is a single cycle. We then determine the fusion product of any $g^s$-twisted module with any $g^r$-twisted module for $1 \leq s, r < k$, where $g = (1, 2, \ldots, k)$, providing an explicit formula for the fusion rules.
Section 5 provides a formula for the fusion product \(\left(T_{g}(V)\right)^{\boxtimes k}\), where \(g = (1, 2, \cdots, k)\). In Section 6, we present an example illustrating the fusion products of twisted modules for the permutation orbifold with $V = V_L$, a rank-1 lattice vertex operator algebra, $k = 4$, and $g = (1, 2, 3, 4)$.

\section{Preliminaries}

\subsection{Basics on vertex operator algebra }

Let $\left(V,Y,\mathbf{1},\omega\right)$ be a vertex operator algebra \cite{FLM,FHL,Bo}. We review twisted modules and related results in this section. 
\begin{definition}
An \emph{automorphism} of a vertex operator algebra $V$ is a linear isomorphism $g$ of $V$ such that $g\left(\omega\right) = \omega$ and $gY\left(v, z\right)g^{-1} = Y\left(gv, z\right)$ for all $v \in V$. Denote by $\text{Aut}\left(V\right)$ the group of all automorphisms of $V$.
\end{definition}

It is well-known that for any subgroup $G$ of $\text{Aut}\left(V\right)$, the set of $G$-fixed points
$$
V^{G} = \left\{ v \in V \mid g(v) = v \ \text{for all}\ g \in G \right\}
$$
is a vertex operator subalgebra.

Let $g$ be a finite-order automorphism of $V$ with order $T$. Then
$$
V = \bigoplus_{r=0}^{T-1} V^{r},
$$
where $V^{r} = \left\{ v \in V \mid gv = e^{2\pi ir/T}v \right\}$ for $r \in \mathbb{Z}$.

\begin{definition}
A \emph{weak $g$-twisted $V$-module} is a vector space $M$ equipped with a linear map
\[
Y_{M}(v, z): V \to \left(\text{End}\,M\right)\left[\left[z^{1/T}, z^{-1/T}\right]\right],
\]
such that
\[
v \mapsto Y_{M}(v, z) = \sum_{n \in \frac{1}{T} \mathbb{Z}} v_{n} z^{-n-1}, \quad v_{n} \in \text{End}\,M,
\]
which satisfies the following conditions for all $u \in V^r$, $v \in V$, and $w \in M$, where $0 \leq r \leq T-1$:
\begin{gather*}
Y_{M}(u, z) = \sum_{n \in \frac{r}{T} + \mathbb{Z}} u_{n} z^{-n-1}, \\
u_{n} w = 0 \quad \text{for $n$ sufficiently large}, \\
Y_{M}(\mathbf{1}, z) = \text{Id}_{M}, \\
z_0^{-1} \delta\left( \frac{z_1 - z_2}{z_0} \right) Y_{M}(u, z_1) Y_{M}(v, z_2) - z_0^{-1} \delta\left( \frac{z_2 - z_1}{-z_0} \right) Y_{M}(v, z_2) Y_{M}(u, z_1) \\
= z_2^{-1} \left( \frac{z_1 - z_0}{z_2} \right)^{-\frac{r}{T}} \delta\left( \frac{z_1 - z_0}{z_2} \right) Y_{M}\left( Y(u, z_0) v, z_2 \right),
\end{gather*}
where $\delta(z) = \sum_{n \in \mathbb{Z}} z^n$.
\end{definition}

\begin{definition}
An \emph{admissible $g$-twisted $V$-module} is a weak $g$-twisted module with a $\frac{1}{T}\mathbb{Z}_{+}$-grading
$$
M = \bigoplus_{n \in \frac{1}{T}\mathbb{Z}_{+}} M(n)
$$
such that for homogeneous $u \in V$ and $m, n \in \frac{1}{T}\mathbb{Z}$, we have
$$
u_{m}M(n) \subset M\left(\text{wt}\,u - m - 1 + n\right).
$$
\end{definition}

Note that if $M = \bigoplus_{n \in \frac{1}{T}\mathbb{Z}_{+}} M(n)$ is an irreducible admissible $g$-twisted $V$-module, then there exists a complex number $\lambda_{M}$ such that $L(0)|_{M(n)} = \lambda_{M} + n$ for all $n$. As a convention, we assume $M(0) \neq 0$, and $\lambda_{M}$ is called the \emph{weight} or \emph{conformal weight} of $M$.

\begin{definition}
A $g$-\emph{twisted $V$-module} is a weak $g$-twisted $V$-module $M$ that carries a $\mathbb{C}$-grading induced by the spectrum of $L(0)$, where $L(0)$ is the component operator of $Y(\omega, z) = \sum_{n \in \mathbb{Z}} L(n) z^{-n-2}$. Specifically, we have
$$
M = \bigoplus_{\lambda \in \mathbb{C}} M_{\lambda},
$$
where $M_{\lambda} = \left\{ w \in M \mid L(0)w = \lambda w \right\}$. Moreover, it is required that $\dim M_{\lambda} < \infty$ for all $\lambda$, and for any fixed $\lambda_0$, $M_{\frac{n}{T} + \lambda_0} = 0$ for all sufficiently small integers $n$.
\end{definition}

When $g = 1$, we recover the notions of weak, ordinary, and admissible $V$-modules (see \cite{DLM2}).

\begin{definition}
A vertex operator algebra $V$ is said to be \emph{$g$-rational} if the admissible $g$-twisted module category is semisimple. In particular, $V$ is said to be \emph{rational} if $V$ is $1$-rational.
\end{definition}

It was proved in \cite{DLM2} that if $V$ is a $g$-rational vertex operator algebra, then there are only finitely many irreducible admissible $g$-twisted $V$-modules up to isomorphism, and any irreducible admissible $g$-twisted $V$-module is ordinary.

\begin{definition}
A vertex operator algebra $V$ is said to be \emph{regular} if every weak $V$-module $M$ is a direct sum of irreducible ordinary $V$-modules.
\end{definition}

\begin{definition}
A vertex operator algebra $V$ is said to be \emph{$C_2$-cofinite} if $V / C_2(V)$ is finite-dimensional, where
$$
C_2(V) = \text{Span}\left\{ u_{-2} v \mid u, v \in V \right\}.
$$
\end{definition}

A vertex operator algebra $V = \bigoplus_{n \in \mathbb{Z}} V_n$ is said to be of \emph{CFT type} if $V_n = 0$ for all negative integers $n$ and $V_0 = \mathbb{C}\mathbf{1}$.

In \cite{DLM2, Li2}, it was demonstrated that if $V$ is a $C_2$-cofinite vertex operator algebra, then $V$ has only finitely many irreducible admissible modules up to isomorphism. For vertex operator algebras of CFT type, regularity is equivalent to rationality and $C_2$-cofiniteness \cite{KL, ABD}.

%%%%%%%%%%%%%%%

\begin{definition}
Let $M = \bigoplus_{n \in \frac{1}{T}\mathbb{Z}_{+}} M(n)$ be an admissible $g$-twisted $V$-module. The \emph{contragredient module} $M'$
is defined as follows:
$$
M' = \bigoplus_{n \in \frac{1}{T}\mathbb{Z}_{+}} M(n)^{*},
$$
where $M(n)^{*} = \text{Hom}_{\mathbb{C}}(M(n), \mathbb{C})$. The vertex operator $Y_{M'}\left(v,z\right)$ is defined for $v\in V$ via
\begin{equation}
\langle Y_{M'}(v, z) f, u \rangle = \langle f, Y_{M}\left(e^{z L(1)} (-z^{-2})^{L(0)} v, z^{-1}\right) u \rangle,
\label{contragredient module}
\end{equation}
where $\langle f, w \rangle = f(w)$ is the natural pairing $M' \times M \to \mathbb{C}$.
\end{definition}

If $M = \bigoplus_{\lambda \in \mathbb{C}} M_{\lambda}$ is a $g$-twisted $V$-module, we define $M' = \bigoplus_{\lambda \in \mathbb{C}} M_{\lambda}^{*}$ and define $Y_{M'}(v, z)$ for $v \in V$ in the same way.

The same argument as in \cite{FHL} yields the following:

\begin{proposition}
If $(M, Y_M)$ is an admissible $g$-twisted $V$-module, then $(M', Y_{M'})$ carries the structure of an admissible $g^{-1}$-twisted $V$-module. Furthermore, if $(M, Y_M)$ is a $g$-twisted $V$-module, then $(M', Y_{M'})$ carries the structure of a $g^{-1}$-twisted $V$-module. Moreover, $M$ is irreducible if and only if $M'$ is irreducible.
\end{proposition}

In particular, if $(M, Y_M)$ is an ordinary $V$-module, then $(M', Y_{M'})$ is also an ordinary $V$-module. An ordinary $V$-module $M$ is said to be \emph{self-dual} if $M$ and $M'$ are isomorphic. A vertex operator algebra $V$ is said to be \emph{self-dual} if $V$ and $V'$ are isomorphic as ordinary $V$-modules.

We will also need the following result from \cite{CM, M2}:

\begin{theorem} \label{CM}
Assume that $V$ is a regular and self-dual vertex operator algebra of CFT type. Then for any solvable subgroup $G$ of $\text{Aut}(V)$, $V^G$ is a regular and self-dual vertex operator algebra of CFT type.
\end{theorem}

%called  the \emph{contragredient module} $M'$. dual
Let $g_1, g_2, g_3$ be mutually commuting automorphisms of $V$ with respective periods $T_1$, $T_2$, and $T_3$. In this case, $V$ decomposes into the direct sum of common eigenspaces for $g_1$ and $g_2$:
$$
V = \bigoplus_{0 \leq j_1 < T_1,\ 0 \leq j_2 < T_2} V^{(j_1, j_2)},
$$
where for $j_1, j_2 \in \mathbb{Z}$,
\begin{align}
V^{(j_1, j_2)} = \left\{ v \in V \mid g_s v = e^{2\pi i j_s / T_s} v,\ s = 1, 2 \right\}. \label{Vij}
\end{align}
For any complex number $\alpha$, we define
$$
(-1)^{\alpha} = e^{\alpha \pi i}.
$$
Now, we define intertwining operators among weak $g_s$-twisted modules $(M_s, Y_{M_s})$ for $s = 1, 2, 3$.

\begin{definition} \label{Intertwining operator for twisted modules}
An \emph{intertwining operator of type} $\binom{M_3}{M_1\ M_2}$ associated with the given data is a linear map 
$$
\mathcal{Y}(\cdot, z): M_1 \to \left(\text{Hom}(M_2, M_3)\right)\{z\}
$$
such that for any $w^1 \in M_1$, $w^2 \in M_2$, and any fixed $c \in \mathbb{C}$,
$$
w_{c+n}^1 w^2 = 0\ \ \text{for $n \in \mathbb{Q}$ sufficiently large},
$$
and
\begin{align}
& z_0^{-1} \left( \frac{z_1 - z_2}{z_0} \right)^{j_1 / T_1} \delta\left( \frac{z_1 - z_2}{z_0} \right) Y_{M_3}(u, z_1) \mathcal{Y}(w, z_2) \nonumber \\
& \quad - z_0^{-1} \left( \frac{z_2 - z_1}{-z_0} \right)^{j_1 / T_1} \delta\left( \frac{z_2 - z_1}{-z_0} \right) \mathcal{Y}(w, z_2) Y_{M_2}(u, z_1) \nonumber \\
= & z_2^{-1} \left( \frac{z_1 - z_0}{z_2} \right)^{-j_2 / T_2} \delta\left( \frac{z_1 - z_0}{z_2} \right) \mathcal{Y}\left( Y_{M_1}(u, z_0) w, z_2 \right), \label{Twisted Intertwining}
\end{align}
on $M_2$ for $u \in V^{(j_1, j_2)}$ with $j_1, j_2 \in \mathbb{Z}$ and $w \in M_1$, and
$$
\frac{d}{dz} \mathcal{Y}(w, z) = \mathcal{Y}(L(-1) w, z).
$$
All intertwining operators of type $\binom{M_3}{M_1\ M_2}$ form a vector space, denoted by $I_V \binom{M_3}{M_1\ M_2}$. The \textit{fusion rule} is given by
$$
N_{M_1 M_2}^{M_3} = \dim I_V \binom{M_3}{M_1\ M_2}.
$$
\end{definition}

\begin{remark} \label{r2.14}
If there exist weak $g_s$-twisted modules $(M_s, Y_{M_s})$ for $s = 1, 2, 3$ such that $N_{M_1 M_2}^{M_3} > 0$, then $g_3 = g_1 g_2$ \cite{X2}. Given this, we will always assume $g_3 = g_1 g_2$.
\end{remark}

We now recall the notion of a tensor product (see \cite{FHL, X2}).

\begin{definition} \label{d2.13}
Let $g_1, g_2$ be commuting, finite-order automorphisms of a vertex operator algebra $V$, and let $M_i$ be a $g_i$-twisted $V$-module for $i = 1, 2$. A \emph{tensor product} for the ordered pair $(M_1, M_2)$ is a pair $(M, F(\cdot, z))$, where $M$ is a weak $g_1 g_2$-twisted $V$-module and $F(\cdot, z)$ is an intertwining operator of type $\binom{M}{M_1\ M_2}$, such that the following universal property holds: For any weak $g_1 g_2$-twisted $V$-module $W$ and any intertwining operator $I(\cdot, z)$ of type $\binom{W}{M_1\ M_2}$, there exists a unique weak twisted $V$-module homomorphism $\psi$ from $M$ to $W$ such that $I(\cdot, z) = \psi \circ F(\cdot, z).$
\end{definition}

We denote by $M_1 \boxtimes_V M_2$ a generic tensor product module of $M_1$ and $M_2$, assuming its existence is verified.

For the remainder of this paper, we assume that $V = \bigoplus_{n \geq 0} V_n$ is a simple, rational, $C_2$-cofinite vertex operator algebra of CFT type, and that $G$ is a finite automorphism group of $V$ such that the conformal weight of any irreducible $g$-twisted $V$-module $M$ is nonnegative, and zero if and only if $M = V$. Under these assumptions, $V^G$ is rational and $C_2$-cofinite if $G$ is solvable \cite{CM, M2}.

\subsection{\label{subsec:Moduriance-invariance-and quantum dimension }Modular
invariance and quantum dimension}

In this subsection, we recall some results on modular invariance in orbifold theory and the notion of the $S$-matrix \cite{Z, DLM3}.

Let $\left(V, Y, \mathbf{1}, \omega\right)$ be a vertex operator algebra with central charge $c$. First, we consider the action of $\text{Aut}(V)$ on the set of twisted modules. Let $g, h \in \text{Aut}(V)$, with $g$ of finite order. If $\left(M, Y_M\right)$ is a weak $g$-twisted $V$-module, there exists a weak $h^{-1}gh$-twisted $V$-module $\left(M \circ h, Y_{M \circ h}\right)$, where $M \circ h \cong M$ as vector spaces and $Y_{M \circ h}(v, z) = Y_M(hv, z)$ for $v \in V$. This defines a right action of $\text{Aut}(V)$ on the set of weak twisted $V$-modules and their isomorphism classes. We say $M$ is \emph{$h$-stable} if $M$ and $M \circ h$ are isomorphic.

Assume that $g$ and $h$ commute. Then $h$ acts on the $g$-twisted modules. Denote by $\mathfrak{U}(g)$ the equivalence classes of irreducible $g$-twisted $V$-modules, and let
$$
\mathfrak{U}(g, h) = \left\{ M \in \mathfrak{U}(g) \mid M \circ h \cong M \right\}.
$$
Both $\mathfrak{U}(g)$ and $\mathfrak{U}(g, h)$ are finite sets, as $V$ is $g$-rational for all $g$.

Let $M$ be an irreducible $g$-twisted $V$-module, and let $G \leq \text{Aut}(V)$ with $G_M = \left\{ h \in G \mid M \circ h \cong M \right\}$. By Schur's Lemma, there is a projective representation $\phi$ of $G_M$ on $M$ such that 
$$
\phi(h) Y(u, z) \phi(h)^{-1} = Y(hu, z),
$$
for $h \in G_M$. If $h = 1$, we take $\phi(1) = 1$. Note that $g$ lies in $G_M$ as $g$ acts naturally on any admissible $g$-twisted module $M$ such that $g|_{M(n)} = e^{2\pi i n}$ for $n \in \frac{1}{T} \mathbb{Z}$. Let $\mathcal{O}_M = \left\{ M \circ g \mid g \in G \right\}$ be the $G$-orbit of $M$. The cardinality of the $G$-orbit $\mathcal{O}_M$ is $\left[ G : G_M \right]$.

Set $o(v) = v_{\text{wt} v - 1}$ for homogeneous $v \in V$. Then $o(v)$ is a degree zero operator of $v$. Let $\mathbb{H}$ be the complex upper half-plane. Here and below, set $q = e^{2\pi i \tau}$, where $\tau \in \mathbb{H}$. For $v \in V$, define
$$
Z_M(v, (g, h), \tau) = \text{tr}_M o(v) \phi(h) q^{L(0) - c/24} = q^{\lambda - c/24} \sum_{n \in \frac{1}{T} \mathbb{Z}_+} \text{tr}_{M_{\lambda + n}} o(v) \phi(h) q^n. \label{def of trace function}
$$
Then $Z_M(v, (g, h), \tau)$ is a holomorphic function on $\mathbb{H}$ (see \cite{Z, DLM3}). For simplicity, we write $Z_M(v, \tau) = Z_M(v, (g, 1), \tau)$. The function $Z_M(\tau) = Z_M(\mathbf{1}, \tau)$ is called the \emph{character} of $M$.

Recall that there is another vertex operator algebra $\left(V, Y[\ \cdot\ ], \mathbf{1}, \tilde{\omega}\right)$ associated with $V$ (see \cite{Z}). Here, $\tilde{\omega} = \omega - c/24$, and for homogeneous $v \in V$,
$$
Y[v, z] = Y(v, e^z - 1) e^{z \cdot \text{wt} v} = \sum_{n \in \mathbb{Z}} v[n] z^{n - 1}.
$$
We write $Y[\tilde{\omega}, z] = \sum_{n \in \mathbb{Z}} L(n) z^{-n-2}$. The weight of a homogeneous $v \in V$ in the second vertex operator algebra is denoted by $\text{wt}[v]$.

The modular group $\Gamma$ consists of $2 \times 2$ integral matrices with determinant 1, denoted as $\Gamma = SL_2(\mathbb{Z})$. Let $\Gamma(N)$ represent the kernel of the reduction modulo $N$ epimorphism $\pi_N: SL_2(\mathbb{Z}) \to SL_2(\mathbb{Z}_N)$. A subgroup $G_N$ of $SL_2(\mathbb{Z})$ is called a \emph{congruence subgroup of level $N$} if $N$ is the least positive integer such that $\Gamma(N) \leq G_N$. Let $P(G)$ be the set of ordered commuting pairs in $G$. For $(g, h) \in P(G)$ and $M \in \mathfrak{U}(g, h)$, $Z_M(v, (g, h), \tau)$ is a function on $V \times \mathbb{H}$. By \cite{DLM3}, the dimension of the vector space $W$ spanned by such functions is equal to
$$
\sum_{(g, h) \in P(G)} \left| \mathfrak{U}(g, h) \right|.
$$
Now we define an action of the modular group $\Gamma$ on $W$ such that
$$
Z_M|_\gamma(v, (g, h), \tau) = (c \tau + d)^{-\text{wt}[v]} Z_M(v, (g, h), \gamma \tau),
$$
where $\gamma: \tau \mapsto \frac{a \tau + b}{c \tau + d}$, and $\gamma = \begin{pmatrix} a & b \\ c & d \end{pmatrix} \in \Gamma = SL_2(\mathbb{Z})$. Let $\gamma \in \Gamma$ act on the right of $P(G)$ via
$$
(g, h) \gamma = (g^a h^c, g^b h^d).
$$

The following theorem is from \cite{Z, DLM3, DLN, DR}:

\begin{theorem}\label{modular invariance thm} 
Let $V$, $G$, and $W$ be as previously defined. Then:

\begin{enumerate}[label=(\arabic*)]

\item There exists a representation $\rho: \Gamma \to GL(W)$ such that for $\left(g, h\right) \in P(G)$, $\gamma = \begin{pmatrix} a & b \\ c & d \end{pmatrix} \in \Gamma$, and $M \in \mathfrak{U}(g, h)$, 
\[
Z_M|_\gamma(v, (g, h), \tau) = \sum_{N \in \mathfrak{U}(g^a h^c, g^b h^d)} \gamma_{M, N} Z_N(v, (g, h), \tau),
\]
where $\rho(\gamma) = (\gamma_{M, N})$. That is,
\[
Z_M(v, (g, h), \gamma \tau) = (c \tau + d)^{\text{wt}[v]} \sum_{N \in \mathfrak{U}(g^a h^c, g^b h^d)} \gamma_{M, N} Z_N(v, (g^a h^c, g^b h^d), \tau).
\]

\item The cardinalities $\left|\mathfrak{U}(g, h)\right|$ and $\left|\mathfrak{U}(g^a h^c, g^b h^d)\right|$ are equal for any $\left(g, h\right) \in P(G)$ and $\gamma \in \Gamma$. In particular, the number of irreducible $g$-twisted $V$-modules exactly equals the number of irreducible $V$-modules that are $g$-stable.

\item Each $Z_M(v, (g, h), \tau)$ is a modular form of weight $\text{wt}[v]$ on the congruence subgroup. In particular, the character $Z_M(\tau)$ is a modular function on the same congruence subgroup.

\end{enumerate}

\end{theorem}

Since the modular group $\Gamma$ is generated by $S = \begin{pmatrix} 0 & -1 \\ 1 & 0 \end{pmatrix}$ and $T = \begin{pmatrix} 1 & 1 \\ 0 & 1 \end{pmatrix}$, the representation $\rho$ is uniquely determined by $\rho(S)$ and $\rho(T)$. The matrix $\rho(S)$ is called the \emph{$S$-matrix} of the orbifold theory. Consider the special case of the $S$-transformation:
\[
Z_M\left(v, -\frac{1}{\tau}\right) = \tau^{\text{wt}[v]} \sum_{N \in \mathfrak{U}(1, g^{-1})} S_{M, N} Z_N(v, (1, g^{-1}), \tau)
\]
for $M \in \mathfrak{U}(g)$, and
\[
Z_N\left(v, (1, g), -\frac{1}{\tau}\right) = \tau^{\text{wt}[v]} \sum_{M \in \mathfrak{U}(g)} S_{N, M} Z_M(v, \tau)
\]
for $N \in \mathfrak{U}(1)$. The matrix $S = (S_{M, N})_{M, N \in \mathfrak{U}(1)}$ is called the \emph{$S$-matrix} of $V$.

Here and below, we denote by $M^0, M^1, \cdots, M^p$ all the inequivalent irreducible $V$-modules, with $M^0 \cong V$. We use $N_{i,j}^k$ to denote $\dim I_V\left(_{M^i\ M^j}^{\ \ \  M^k}\right)$, where $I_V\left(_{M^i\ M^j}^{\ \ \ M^k}\right)$ is the space of intertwining operators of type $\left(_{M^i\ M^j}^{\ \ \ M^k}\right)$. The numbers $N_{i,j}^k$ are called the \emph{fusion rules}. As usual, we use $M^{i'}$ to denote the contragredient module of $M^i$. We will denote $S_{M^i, M^j}$ by $S_{i, j}$ for $0 \le i, j \le p$. The following Verlinde formula \cite{V} is proved in \cite{H2}.

\begin{theorem} \label{Verlinde formula}
Let $V$ be a simple, rational, $C_2$-cofinite vertex operator algebra of CFT type, and assume $V \cong V'$. Let $S = (S_{i, j})_{i, j=0}^{p}$ be the $S$-matrix of $V$. Then:

\begin{enumerate}[label=(\arabic*)]

\item $\left(S^{-1}\right)_{i, j} = S_{i, j'} = S_{i', j}$ and $S_{i', j'} = S_{i, j}.$

\item $S$ is symmetric and $S^2 = (\delta_{i, j'}).$

\item $N_{i, j}^{k} = \sum_{t=0}^{p} \frac{S_{j, t} S_{i, t} S_{t, k}^{-1}}{S_{0, t}}.$

\item The $S$-matrix diagonalizes the fusion matrix $N(i) = \left(N_{i, j}^{k}\right)_{j, k = 0}^{p}$, with diagonal entries $\frac{S_{i, t}}{S_{0, t}}$ for $i, t = 0, \cdots, p$. More explicitly, $S^{-1} N(i) S = \text{diag}\left(\frac{S_{i, t}}{S_{0, t}}\right)_{t=0}^{p}.$ In particular, $S_{0, t} \neq 0$ for any $t = 0, 1, \cdots, p$.

\end{enumerate}

\end{theorem}

The notion of quantum dimension was introduced in \cite{DJX}.

\begin{definition}
Suppose $Z_V(z)$ and $Z_M(z)$ exist. The quantum dimension of $M$ over $V$ is defined as:
\[
\text{qdim}_V M = \lim_{y \to 0^+} \frac{Z_M(iy)}{Z_V(iy)},
\]
where $y$ is real and positive.
\end{definition}

We will use the following result from \cite[Lemma 4.2]{DJX}:

\begin{lemma}
Let $V$ be a simple, rational, $C_2$-cofinite vertex operator algebra of CFT type, and assume $V \cong V'$. Let $\text{Irr}(V) = \left\{ M^0 = V, M^1, \cdots, M^p \right\}$ be the set of all inequivalent irreducible $V$-modules. Then $0 < \text{qdim}_V M^i < \infty$, and
\[
\text{qdim}_V M^i = \frac{S_{i, 0}}{S_{0, 0}},
\]
for any $0 \le i \le p$.
\end{lemma}

\subsection{Basics on category theory }

Let $\mathbb{K}$ be an algebraically closed field of characteristic zero. Below, we recall some basic concepts from category theory \cite{BK, CKM, KO, Mu}.

\begin{definition}
A \emph{fusion category} is a $\mathbb{K}$-linear, semisimple, rigid tensor category $\mathcal{C}$ with finitely many isomorphism classes of simple objects, finite-dimensional morphism spaces, and such that the unit object $1$ in $\mathcal{C}$ is simple.
\end{definition}

Let $\mathcal{C}$ be a fusion category with simple objects $X_1, \dots, X_n$. One can define non-negative integers $a_{ij}^k \in \mathbb{N}$ via the relation
\[
X_i \otimes X_j \cong \bigoplus_{k=1}^{n} a_{ij}^{k} X_k.
\]

\begin{definition}
A fusion category is \emph{braided} if it is equipped with a natural isomorphism $c_{X, Y}: X \otimes Y \cong Y \otimes X$ that satisfies the hexagon axioms.
\end{definition}

\begin{definition}
A \emph{modular tensor category} is a semisimple ribbon category $\mathcal{C}$ that satisfies the following properties:
\begin{enumerate}[label=(\roman*)]
    \item $\mathcal{C}$ has a finite number of isomorphism classes of simple objects.
    \item The matrix $\tilde{s} = (\tilde{s}_{ij})_{i, j \in I}$ is invertible.
\end{enumerate}
\end{definition}

\begin{definition}
Let $\left(\mathcal{C}, \boxtimes, 1, l, r, a, \mathcal{R}\right)$ be a braided tensor category with tensor unit $1$, braiding $\mathcal{R}$, and associativity isomorphism $a$. A \emph{commutative associative algebra} in $\mathcal{C}$ is an object $A$ of $\mathcal{C}$ equipped with morphisms $\mu: A \boxtimes A \to A$ and $\iota_A: 1 \to A$ such that:
\begin{enumerate}[label=(\arabic*)]
    \item \emph{Associativity}: The compositions $\mu \circ (\mu \boxtimes 1_A) \circ a_{A, A, A}$ and $\mu \circ (1_A \boxtimes \mu)$ from $A \boxtimes (A \boxtimes A) \to A$ are equal.
    \item \emph{Commutativity}: The morphisms $\mu$ and $\mu \circ \mathcal{R}_{A, A}: A \boxtimes A \to A$ are equal.
    \item \emph{Unit}: The composition $\mu \circ (\iota_A \boxtimes 1_A) \circ \iota_A^{-1}: A \to A$ is the identity.
\end{enumerate}
\end{definition}

The following result is from \cite{Lu}:

\begin{theorem} \label{Lustig fusion algebra}
Let $\mathcal{C}$ be a fusion category with simple objects $X_1, \dots, X_n$. Let $K(\mathcal{C}) = \bigoplus_{i=1}^{n} \mathbb{C} X_i$ with $X_i X_j = \sum a_{ij}^k X_k$ be its fusion ring, where $a_{ij}^k$ are the fusion rules. Then $K(\mathcal{C})$ is a semisimple associative algebra.
\end{theorem}

\subsection{The category $\text{Rep}\left(V\right)$ and fusion rules}

In this section, we introduce the category $\text{Rep}(V)$ and establish some properties of the fusion rules for twisted $V$-modules. First, we recall the following result from \cite{H3}:

\begin{theorem}
If $V$ is a simple, rational, $C_2$-cofinite, CFT-type, self-dual vertex operator algebra, then $\mathcal{C}_V$ is a modular tensor category.
\end{theorem}

Assume that $G$ is a finite abelian automorphism group of $V$. Let $\mathcal{C}_V$ and $\mathcal{C}_{V^G}$ represent the categories of ordinary $V$-modules and ordinary $V^G$-modules, respectively. Both $\mathcal{C}_V$ and $\mathcal{C}_{V^G}$ are modular tensor categories \cite{H3}. Furthermore, $V$ is a commutative associative algebra in $\mathcal{C}_{V^G}$ \cite{KO, CKM}, where $V = \bigoplus_{\chi \in \text{Irr}(G)} V^\chi$, and $\text{Irr}(G)$ is the set of irreducible characters of $G$, with $V^\chi$ representing irreducible $V^G$-modules (i.e., simple objects in $\mathcal{C}_{V^G}$).

Now, we recall the following definition from \cite{KO, CKM}:

\begin{definition}\label{Rep(V)-def}
Denote by $\text{Rep}(V)$ the subcategory of $\mathcal{C}_{V^G}$ consisting of every $V^G$-module $W$, along with a $V^G$-intertwining operator $Y_W(\cdot, z)$ of type $\left(_{V W}^{\ W}\right)$, such that the following conditions are satisfied:

\begin{enumerate}[label=(\arabic*)]
    \item \textbf{(Associativity)} For any $u, v \in V$, $w \in W$, and $w' \in W'$, the formal series
    \[
    \left\langle w', Y_W(u, z_1) Y_W(v, z_2) w \right\rangle
    \]
    and
    \[
    \left\langle w', Y_W(Y(u, z_1 - z_2)v, z_2) w \right\rangle
    \]
    converge on the domains $\left|z_1\right| > \left|z_2\right| > 0$ and $\left|z_2\right| > \left|z_1 - z_2\right| > 0$, respectively, to multivalued analytic functions that coincide on their common domain.

    \item \textbf{(Unit)} $Y_W(\mathbf{1}, z) = \text{Id}_W$.
\end{enumerate}
\end{definition}

Note that the objects of $\text{Rep}(V)$ are more general than ordinary $V$-modules, as we do not require the $V$-intertwining operator $Y_W(\cdot, x)$ for a module $W$ in $\text{Rep}(V)$ to be restricted to integral powers of the formal variable $x$. The following result is given in \cite{K1, K2, DLXY}:

\begin{lemma}\label{l3.1}
\begin{enumerate}[label=(\arabic*)]
    \item If $W$ is a $g$-twisted $V$-module with $g \in G$, then $W$ is an object of $\text{Rep}(V)$.
    \item If $W$ is a simple object in $\text{Rep}(V)$, then $W$ is an irreducible $g$-twisted $V$-module for some $g \in G$.
\end{enumerate}
\end{lemma}

Therefore, we have:
\[
\text{Rep}(V) = \bigoplus_{g \in G} \mathcal{C}_{V, g},
\]
where $\mathcal{C}_{V, g}$ is the category of $g$-twisted $V$-modules. The following result can be found in \cite{KO, CKM}:

\begin{lemma}\label{Rep(V)-FusionAlg}
The category $\text{Rep}(V)$ is a fusion category.
\end{lemma}

Note that  there is another tensor product in Definition \ref{d2.13} on 
$\text{Rep}(V).$ It is shown in \cite{DLXY} that these two  tensor products are the same. In particular, the tensor product defined in  Definition \ref{d2.13} is also associative.  

We will need the following symmetric property of fusion rules for twisted modules.

\begin{proposition}\label{Fusion rule prop twisted}
Let $g_1, g_2$ be commuting finite-order automorphisms of a vertex operator algebra $V$. Let $M_i$ be a $g_i$-twisted $V$-module for $i = 1, 2, 3$, where $g_3 = g_1 g_2$. Then:
\begin{gather}
N_{M_1, M_2}^{M_3} = N_{M_2 \circ g_1, M_1}^{M_3} = N_{M_2, M_1 \circ g_2^{-1}}^{M_3} = N_{M_1 \circ g_2^{-1}, (M_3)'}^{(M_2)'}.\label{sym prop of fusion rules}
\end{gather}
\end{proposition}

\begin{proof}
The first two identities are given in Proposition 3.6 of \cite{DL}. Now we prove the last equality. Note that $M_3'$ is a $g_1^{-1}g_2^{-1}$-twisted module. We have:
\begin{alignat*}{1}
\text{Hom}(M_1 \boxtimes M_2, M_3) & \cong \text{Hom}(M_1, M_3 \boxtimes M_2') \\
& \cong \text{Hom}(M_3' \boxtimes M_1, M_2') \\
& \cong \text{Hom}\left((M_1 \circ g_1^{-1} g_2^{-1}) \boxtimes M_3', M_2'\right) \\
& \cong \text{Hom}\left((M_1 \circ g_2^{-1}) \boxtimes M_3', M_2'\right),
\end{alignat*}
where the first two isomorphisms follow from the properties of fusion categories \cite{EGNO}, the third isomorphism relies on the first equality in (\ref{sym prop of fusion rules}), and the last isomorphism utilizes the fact that $M_1 \circ g^{-1} = M_1$.
\end{proof}

In particular, if $g_1 = g_2 = 1$, then $M_1$, $M_2$, and $M_3$ are $V$-modules. In this case, we recover the symmetry properties of fusion rules for $V$-modules, as given in \cite{FHL}:

\begin{equation}
N_{M_1, M_2}^{M_3} = N_{M_2, M_1}^{M_3}, \quad N_{M_1, M_2}^{M_3} = N_{M_1, (M_3)'}^{(M_2)'}.\label{Fusion rule property for untwsted modules}
\end{equation}

\section{Permutation Orbifolds}

In the rest of this paper, we fix $g=\left(1,2,\cdots,k\right)$ where
$k$ is a positive integer. In this section, we first review the structure
of twisted modules in permutation orbifold from \cite{BDM} and give
a result on the dual module of a twisted module in permutation orbifold.
Then we review results on $S$-matrix and fusion product of an untwisted
module with a twisted module in permutation orbifolds from \cite{DXY4,DLXY}
, which will be applied in subsequent discussions.

\subsection{Twisted modules in permutation orbifolds}

A fundamental connection between twisted modules for tensor
product vertex operator algebra $V^{\otimes k}$ with respect to permutation
automorphisms and $V$-modules was found in \cite{BDM}. More specifically,
for any $V$-module $\left(W,Y_{W}(\cdot,z)\right)$, a canonical
$g$-twisted $V^{\otimes k}$-module structure $T_{g}(W)$ on $W$
was obtained, where $g$ is viewed as an automorphism of $V^{\otimes k}$.
Moreover, it was demonstrated therein that this leads to an isomorphism
between the categories of weak, admissible and ordinary $V$-modules
and the categories of weak, admissible and ordinary $g$-twisted $V^{\otimes k}$-modules,
respectively. In this subsection, we first recall structure of twisted
modules in permutation orbifold from \cite{BDM}. Following that,
we prove a result concerning the contragredient module of twisted
modules in permutation orbifolds.

Let $k$ be a positive integer. Recall from \cite{BDM} that 
\begin{align}
\Delta_{k}(z)=\exp\left(\sum_{n\ge1}a_{n}z^{-\frac{n}{k}}L(n)\right)k^{-L(0)}z^{(1/k-1)L(0)},\label{delta-kx}
\end{align}
where the coefficients $a_{n}$ for $n\ge1$ are uniquely determined
by 
\begin{align}
\exp\left(\sum_{n\ge1}-a_{n}x^{n+1}\frac{d}{dx}\right)x=\frac{1}{k}(1+x)^{k}-\frac{1}{k}.\label{adefine}
\end{align}

The operator $\Delta_{k}\left(z\right)$ satisfies the following property
\cite{BDM}:

\begin{proposition} In $\left(\text{End}V\right)\left[\left[z^{1/k},z^{-1/k}\right]\right]$,
we have

\[
\Delta_{k}\left(z\right)Y\left(u,z_{0}\right)\Delta_{k}\left(z\right)^{-1}=Y\left(\Delta_{k}\left(z+z_{0}\right)u,\left(z+z_{0}\right)^{1/k}-z^{1/k}\right),
\]
for all $u\in V$.

\end{proposition}

The following result is from \cite{BDM}:

\begin{theorem}Let $V$ be any vertex operator algebra and let $g=(1, 2,\cdots, k)\in\text{Aut}(V^{\otimes k})$.
Then for any (weak, admissible) $V$-module $(W,Y_{W})$, we have
a (weak, admissible) $g$-twisted $V^{\otimes k}$-module $\left(T_{g}(W),Y_{T_{g}(W)}\right)$,
where $T_{g}(W)=W$ as a vector space and the vertex operator map
$Y_{T_{g}(W)}(\cdot, z)$ is uniquely determined by

\begin{equation}
Y_{T_{g}(W)}(u^{1},z)=Y_{W}\left(\Delta_{k}(z)u,z^{1/k}\right),\ \ \mbox{ for }u\in V,\label{twisted module delta}
\end{equation}
where $u^{1}=u\otimes{\bf 1}^{\otimes(k-1)}\in V^{\otimes k}$. Furthermore,
every (weak, admissible) $g$-twisted $V^{\otimes k}$-module is isomorphic
to one in this form. \end{theorem}

Here, we present a straightforward result in formal calculus that
will be utilized in subsequent discussions.

\begin{lemma} \label{lemma on binomial coefficients}For $m\ge1$,
\[
\sum_{i=1}^{k}\left(_{m-i}^{m-1}\right)\left(_{i}^{k}\right)=\left(_{m}^{k+m-1}\right).
\]

\end{lemma}
\begin{proof}
We consider the coefficient of $z^{m}$ in $\left(1+z\right)^{k+m-1}.$
Since $\left(1+z\right)^{k+m-1}=\sum_{i=0}^{k+m-1}\left(_{i}^{k+m-1}\right)z^{i},$
the coefficient of $z^{m}$ is $\left(_{m}^{k+m-1}\right)$. On the
other hand, from 
\[
\left(1+z\right)^{k+m-1}=\left(1+z\right)^{m-1}\left(1+z\right)^{k}=\sum_{i=0}^{m-1}\left(_{i}^{m-1}\right)z^{i}\cdot\sum_{j=0}^{k}\left(_{j}^{k}\right)z^{j},
\]
we see that the coefficient of $z^{m}$ is $\left(_{m-i}^{m-1}\right)\left(_{i}^{k}\right)$.
Therefore, $\sum_{i=1}^{k}\left(_{m-i}^{m-1}\right)\left(_{i}^{k}\right)=\left(_{m}^{k+m-1}\right).$ 
\end{proof}

\begin{lemma}\label{l3.4} We have
\begin{equation}
e^{z^{\frac{1}{k}}L(1)}\left(-z^{-\frac{2}{k}}\right)^{L\left(0\right)}\Delta_{k}\left(z\right)=\Delta_{k}\left(z^{-1}\right)e^{zL(1)}\left(-z^{-2}\right)^{L\left(0\right)}\label{key identity}
\end{equation}
where $\Delta_{k}\left(z^{-1}\right)=\exp\left(\sum_{n\ge1}a_{n}z^{\frac{n}{k}}L\left(n\right)\right)k^{L\left(0\right)}z^{\left(1-\frac{1}{k}\right)L\left(0\right)}$.
\end{lemma}
\begin{proof}
    We consider faithful representation on $\mathbb{C}\left[x,x^{-1}\right]$:
$L(n)\to-x^{n+1}\frac{d}{dx}$. So it is good enough to show 
\begin{flalign}
 & \exp\left(z^{\frac{1}{k}}\left(-x^{2}\frac{d}{dx}\right)\right)\cdot\left(-z^{-\frac{2}{k}}\right)^{-x\frac{d}{dx}}\cdot \exp\left(-\sum_{j\ge1}a_{j}z^{-j/k}x^{j+1}\frac{d}{dx}\right)\cdot k^{x\frac{d}{dx}}z^{-\left(\frac{1}{k}-1\right)x\frac{d}{dx}}x\nonumber \\
= & \Delta_{k}\left(z^{-1}\right)\exp\left(-zx^{2}\frac{d}{dx}\right)\left(-z^{-2}\right)^{-x\frac{d}{dx}}x.\label{identity to be proved}
\end{flalign}
It is easy to check that 
\begin{equation}
\left(x^{2}\frac{d}{dx}\right)^{n}x^{i}=i\cdots\left(i+n-1\right)x^{i+n}\ \text{for\ }i\ge1,n\ge1.\label{operator on x^i}
\end{equation}
In particular, 
\begin{equation}
\left(x^{2}\frac{d}{dx}\right)^{n}x=n!\cdot x^{n+1}.\label{r.s.}
\end{equation}

The left-hand side of (\ref{identity to be proved}) equals 
\begin{alignat}{1}
 & \exp\left(z^{\frac{1}{k}}\left(-x^{2}\frac{d}{dx}\right)\right)\cdot\left(-z^{-\frac{2}{k}}\right)^{-x\frac{d}{dx}}\cdot \exp\left(-\sum_{j\ge1}a_{j}z^{-\frac{j}{k}}x^{j+1}\frac{d}{dx}\right)\cdot k^{x\frac{d}{dx}}z^{-\left(\frac{1}{k}-1\right)x\frac{d}{dx}}x\nonumber \\
= & \exp\left(-z^{\frac{1}{k}}x^{2}\frac{d}{dx}\right)\cdot\left(-z^{-\frac{2}{k}}\right)^{-x\frac{d}{dx}}\cdot z\left(\left(1+z^{-\frac{1}{k}}x\right)^{k}-1\right)\nonumber \\
= & \exp\left(-z^{\frac{1}{k}}x^{2}\frac{d}{dx}\right)\cdot\left(-z^{-\frac{2}{k}}\right)^{-x\frac{d}{dx}}\cdot z\cdot\left(\sum_{i=0}^{k}\left(_{i}^{k}\right)z^{-\frac{i}{k}}x^{i}-1\right)\nonumber \\
= & \exp\left(-z^{\frac{1}{k}}x^{2}\frac{d}{dx}\right)\cdot\sum_{i=1}^{k}\left(_{i}^{k}\right)\left(-z^{\frac{2}{k}}\right)^{-i}z^{-\frac{i}{k}+1}x^{i}\nonumber \\
= & \exp\left(-z^{\frac{1}{k}}x^{2}\frac{d}{dx}\right)\cdot\sum_{i=1}^{k}\left(_{i}^{k}\right)z^{\frac{i}{k}+1}\left(-1\right)^{i}x^{i}\nonumber \\
= & \sum_{n\ge0}\frac{\left(-z^{\frac{1}{k}}\right)^{n}}{n!}\left(x^{2}\frac{d}{dx}\right)^{n}\cdot\sum_{i=1}^{k}\left(_{i}^{k}\right)z^{\frac{i}{k}+1}\left(-1\right)^{i}x^{i}\nonumber \\
= & \sum_{n\ge0}\frac{\left(-z^{\frac{1}{k}}\right)^{n}}{n!}\cdot\sum_{i=1}^{k}\left(_{i}^{k}\right)z^{\frac{i}{k}+1}\left(-1\right)^{i}x^{i+n}i\cdots\left(i+n-1\right)\nonumber \\
= & \sum_{n\ge0}\frac{\left(-z^{\frac{1}{k}}x\right)^{n}}{n!}\cdot\sum_{i=1}^{k}\left(_{i}^{k}\right)\left(z^{\frac{1}{k}}x\right)^{i}z\left(-1\right)^{i}i\cdots\left(i+n-1\right)\label{left-1}
\end{alignat}
where in the first identity we use results from the proof in \cite[Lemma 5.4]{DLXY}
and in the third to the last identity we use (\ref{operator on x^i}).

Recall from \cite{BDM} that $\Delta_{k}\left(z\right)x^{n}=\left(\Delta_{k}\left(z\right)x\right)^{n}$for
$n\in\mathbb{Z}$. The right-hand side of (\ref{identity to be proved}) is equal to
\begin{flalign}
 & \Delta_{k}\left(z^{-1}\right)\exp\left(-zx^{2}\frac{d}{dx}\right)\left(-z^{-2}\right)^{-x\frac{d}{dx}}x\nonumber \\
= & \Delta_{k}\left(z^{-1}\right)\sum_{n\ge0}\frac{\left(-z\right)^{n}}{n!}\left(x^{2}\frac{d}{dx}\right)^{n}x\left(-z^{-2}\right)^{-1}\nonumber \\
= & \Delta_{k}\left(z^{-1}\right)\sum_{n\ge0}\frac{\left(-z\right)^{n}}{n!}n!x^{n+1}\left(-1\right)z^{2}\nonumber \\
= & \Delta_{k}\left(z^{-1}\right)\sum_{n\ge0}\left(-1\right)^{n+1}z^{n+2}x^{n+1}\nonumber \\
= & \sum_{n\ge0}\left(-1\right)^{n+1}z^{n+2}\Delta_{k}\left(z^{-1}\right)x^{n+1}\nonumber \\
= & \sum_{n\ge0}\left(-1\right)^{n+1}z^{n+2}\left(\Delta_{k}\left(z^{-1}\right)x\right)^{n+1}\nonumber \\
= & \sum_{n\ge0}\left(-1\right)^{n+1}z^{n+2}\left(z^{-1}\left(\left(1+z^{1/k}x\right)^{k}-1\right)\right)^{n+1}\nonumber \\
= & \sum_{n\ge0}\left(-1\right)^{n+1}z\left(\left(1+z^{1/k}x\right)^{k}-1\right)^{n+1}\label{right-1}
\end{flalign}
where in the second identity we use (\ref{r.s.}) and in the sixth
identity we use results from the proof in \cite[Lemma 5.4]{DLXY}.

So it is good enough to show
\begin{equation}
\sum_{n\ge0}\frac{\left(-z^{\frac{1}{k}}x\right)^{n}}{n!}\cdot\sum_{i=1}^{k}\left(_{i}^{k}\right)\left(z^{\frac{1}{k}}x\right)^{i}\left(-1\right)^{i}i\cdots\left(i+n-1\right)=\sum_{n\ge0}\left(-1\right)^{n+1}\left(\left(1+z^{\frac{1}{k}}x\right)^{k}-1\right)^{n+1}.\label{Eq ID-2}
\end{equation}
Let $y=z^{\frac{1}{k}}x,$ then the left-hand side of (\ref{Eq ID-2}) is equal
to
\begin{flalign}
 & \sum_{n\ge0}\frac{\left(-y\right)^{n}}{n!}\cdot\sum_{i=1}^{k}\left(_{i}^{k}\right)y^{i}\left(-1\right)^{i}i\cdots\left(i+n-1\right)\nonumber \\
= & \sum_{n\ge0}\sum_{i=1}^{k}\frac{\left(-y\right)^{n+i}i\cdots\left(i+n-1\right)}{n!}\left(_{i}^{k}\right)\nonumber \\
= & \sum_{n\ge0}\sum_{i=1}^{k}\left(_{n}^{n+i-1}\right)\left(_{i}^{k}\right)\left(-y\right)^{n+i}\nonumber \\
= & \sum_{m\ge1}\sum_{i=1}^{k}\left(_{m-i}^{m-1}\right)\left(_{i}^{k}\right)\left(-y\right)^{m},\label{left-2}
\end{flalign}
where in the last identity we let $m=n+i.$ 
The right-hand side of (\ref{Eq ID-2}) is equal to 
\begin{flalign}
 & \sum_{n\ge0}\left(-1\right)^{n+1}\left(\left(1+y\right)^{k}-1\right)^{n+1}\nonumber \\
= & \frac{\left(-1\right)\left(\left(1+y\right)^{k}-1\right)}{1+\left(\left(1+y\right)^{k}-1\right)}\\
= & -1+\frac{1}{\left(1+y\right)^{k}}\nonumber \\
= & -1+\sum_{i\ge0}\left(_{i}^{-k}\right)y^{i}\nonumber \\
= & -1+\sum_{i\ge0}\frac{\left(-k\right)\left(-k-1\right)\cdots\left(-k-i+1\right)}{i!}y^{i}\nonumber \\
= & -1+\sum_{m\ge0}\frac{k\left(k+1\right)\cdots\left(k+m-1\right)}{m!}\left(-1\right)^{m}y^{m}\nonumber \\
= & \sum_{m\ge1}\left(_{m}^{k+m-1}\right)\left(-y\right)^{m}.\label{right-2}
\end{flalign}
By Lemma \ref{lemma on binomial coefficients}, (\ref{left-2}) is
equal to (\ref{right-2}) and hence the proof is completed. 
\end{proof}
\begin{proposition} \label{twisted dual}Let $V$ be any vertex operator
algebra and set $g=(1, 2,\cdots, k)\in\text{Aut}\left(V^{\otimes k}\right)$
with $k$ a positive integer. Let $W$ any $V$-module. Denote $T_{g}\left(W\right)'$
the restrict dual of the $g$-twisted $V^{\otimes k}$-module $T_{g}\left(W\right)$.
Then $T_{g}\left(W\right)'=T_{g^{-1}}\left(W'\right)$ where $W'$
is the contragredient module of $W.$

\end{proposition} 
\begin{proof}
For any   $v\in V$, $v^1=v\otimes 1^{\otimes(k-1)}\in V^{\otimes k}$, $w\in W$ and $w'\in W'$, by the definition of $W'$ from (\ref{contragredient module})
we have 
$$
\left\langle Y_{W'}\left(v,z^{\frac{1}{k}}\right)w',w\right\rangle =\left\langle w',Y_W\left(e^{z^{\frac{1}{k}}L(1)}\left(-z^{-\frac{2}{k}}\right)^{L\left(0\right)}v,z^{-\frac{1}{k}}\right)w\right\rangle$$
and 
$$
\left\langle Y_{W'}\left(\Delta_k(z)v,z^{\frac{1}{k}}\right)w',w\right\rangle =\left\langle w',Y_W\left(e^{z^{\frac{1}{k}}L(1)}\left(-z^{-\frac{2}{k}}\right)^{L\left(0\right)}\Delta_k(z)v,z^{-\frac{1}{k}}\right)w\right\rangle.$$
Using Lemma \ref{l3.4} we see that
$$\left\langle Y_{W'}\left(\Delta_k(z)v,z^{\frac{1}{k}}\right)w',w\right\rangle =\left\langle w',Y_W\left(\Delta_k(z^{-1})e^{zL(1)}(z^{-2})^{L(0)}v,z^{-\frac{1}{k}}\right)w\right\rangle,$$
or equivalently,
$$\left\langle Y_{T_{g^{-1}(W')}}\left(v^1,z\right)w',w\right\rangle =\left\langle w',Y_{T_{g}\left(W\right)}\left(e^{zL(1)}(z^{-2})^{L(0)}v^1,z^{-1}\right)w\right\rangle.$$
Noting that $T_g(W)'=W'$ as vector spaces and 
$$\left\langle Y_{T_{g}\left(W\right)'}\left(v^1,z\right)w',w\right\rangle =\left\langle w',Y_{T_{g}\left(W\right)}\left(e^{zL(1)}(z^{-2})^{L(0)}v^1,z^{-1}\right)w\right\rangle,$$
we conclude that 
$$Y_{T_{g^{-1}(W')}}\left(v^1,z\right)=Y_{T_{g}\left(W\right)'}\left(v^1,z\right)$$
for all $v.$ That is,  $T_{g}\left(W\right)'=T_{g^{-1}}\left(W'\right),$ as expected. 
\end{proof}
\begin{remark} If we further assume the vertex operator algebra $V$
in the Proposition \ref{twisted dual} satisfies that $T_{g}\left(V\right)'=T_{g^{-1}}\left(V\right).$
Then we have a simpler proof of the above proposition. Indeed, since
$T_{g}\left(V\right)'=T_{g^{-1}}\left(V\right),$ one has $N_{T_{g}\left(V\right)',T_{g^{-1}}\left(V\right)}^{V}\ge1$.
Also it is clear that $N_{W',W}^{V}\ge1$. Thus 
\begin{alignat*}{1}
 & \text{Hom}\left(T_{g}\left(W\right)\boxtimes_{V^{\otimes k}}T_{g^{-1}}\left(W'\right),V^{\otimes k}\right)\\
= & \text{Hom}\left(T_{g}\left(V\right)\boxtimes_{V^{\otimes k}}\left(W'\otimes V^{k-1}\right)\boxtimes_{V^{\otimes k}}\left(W\otimes V^{k-1}\right)\boxtimes_{V^{\otimes k}}T_{g^{-1}}\left(V\right),V^{\otimes k}\right)\\
= & \text{Hom}\left(\left(W'\boxtimes_{V}W\right)\boxtimes_{V^{\otimes k}}V^{k-1},T_{g^{-1}}\left(V\right)\boxtimes T_{g}\left(V\right)\right)\ge1,
\end{alignat*}
which implies $T_{g}\left(W\right)'=T_{g^{-1}}\left(W'\right)$ by
properties of fusion rules.

\end{remark}

\subsection{S-matrix in permutation orbifolds}

In this subsection, we review some key properties of the $S$-matrix in permutation orbifolds from \cite{DXY4} that will be essential for determining the fusion products of twisted modules in permutation orbifolds.

Let $1 \leq r, s < k$, with $d = \text{gcd}(s, k)$ and $f = \text{gcd}(d, r)$. Set $l = \frac{k}{d}$ and $c = \frac{d}{f}$. The equivalence classes of irreducible $g^s$-twisted $V^{\otimes k}$-modules that are $g^r$-stable are given by:
\begin{equation}
\mathfrak{U}(g^s, g^r) = \left\{ \left(T_{\alpha_1}(W^1)\right)^{\otimes c} \otimes \cdots \otimes \left(T_{\alpha_f}(W^f)\right)^{\otimes c} \mid W^1, \dots, W^f \in \text{Irr}(V) \right\},\label{general module g^s,g^r}
\end{equation}
where each $\alpha_i$ is an $l$-cycle for $1 \leq i \leq f$. We will denote modules of the form $\left(T_{\alpha_1}(W^1)\right)^{\otimes c} \otimes \cdots \otimes \left(T_{\alpha_f}(W^f)\right)^{\otimes c}$ by $T_{g^s}^{W^1, \dots, W^f; c}$, as in \cite{DXY4}.

Let $\mathfrak{U}(g^r, g^s)$ represent the equivalence classes of irreducible $g^r$-twisted modules that are $g^s$-stable. Assume that $\text{gcd}(r, k) = d_1$. Then there exist $m, n \in \mathbb{Z}$ such that $rm + kn = d_1$. Note also that $\text{gcd}(d_1, s) = \text{gcd}(d, r) = \text{gcd}(s, k, r) = f$. Set $l_1 = \frac{k}{d_1}$ and $a = \frac{d_1}{f}$. Similarly, any $g^r$-twisted module in $\mathfrak{U}(g^r, g^s)$ can be written in the form
\[
T_{g^r}^{W^1, \dots, W^f; a} = \left(T_{\beta_1}(W^1)\right)^{\otimes a} \otimes \cdots \otimes \left(T_{\beta_f}(W^f)\right)^{\otimes a}, \label{general module g^r,g^s}
\]
where $W^i \in \text{Irr}(V)$, and each $\beta_i$ is an $l_1$-cycle for $1 \leq i \leq f$.

Now, we recall the modules in $\mathfrak{U}(1, g^s)$. Let $i_1, \dots, i_k \in \{0, 1,  \dots, p\}$. Set $M^{i_1, \dots, i_k} = M^{i_1} \otimes \cdots \otimes M^{i_k}$. Then $M^{i_1, \dots, i_k}$ is an irreducible $V^{\otimes k}$-module. As noted in \cite[(3.1)]{DXY4}, $g^s$ can be expressed as a product of $d$ disjoint $l$-cycles: $g^s = \tau_1 \cdots \tau_d$, where $d = \text{gcd}(s, k)$ and $k = dl$. Therefore, $M^{i_1, \dots, i_k} \circ g^s \cong M^{i_1, \dots, i_k}$ if and only if $M^{i_a} = M^{i_a + jd}$ for $a = 1, \dots, d$ and $j = 1, \dots, l - 1$. Thus, any $V^{\otimes k}$-module in $\mathfrak{U}(1, g^s)$ can be written in the form:
\[
\left(M^{i_1}\right)^{\otimes l} \otimes \cdots \otimes \left(M^{i_d}\right)^{\otimes l}, \label{1,g^s}
\]
where $i_1, \dots, i_d \in \{0, 1, \dots, p\}$.

We will sometimes denote a $g^s$-twisted module of the form $T_{\tau_1}(M^{j_1}) \otimes \cdots \otimes T_{\tau_d}(M^{j_d})$ by $T_{g^s}^{j_1, \dots, j_d}$ or $T_{\tau_1, \dots, \tau_d}^{j_1, \dots, j_d}$ with $j_1, \dots, j_d \in \{0, 1, \dots, p\}$. The following lemma is given in \cite[Lemma 4.5]{DXY4}.

\begin{lemma} \label{S-matrix, nontwisted and twisted} Suppose $s\in\mathbb{N},$
$d=\text{gcd}\left(s,k\right)$ and $l=\frac{k}{d}.$ Let $\left(M^{i_{1}}\right)^{\otimes l}\otimes\cdots\otimes\left(M^{i_{d}}\right)^{\otimes l}\in\mathcal{\mathfrak{U}}\left(1,g^{s}\right)$
and $T_{g^{s}}^{j_{1},\cdots,j_{d}}$ be a $g^{s}$-twisted $V^{\otimes k}$-module
as above. Then 
\[
S_{\left(M^{i_{1}}\right)^{\otimes l}\otimes\cdots\otimes\left(M^{i_{d}}\right)^{\otimes l},T_{g^{s}}^{j_{1},\cdots,j_{d}}}=S_{M^{i_{1}},M^{j_{1}}}\cdots S_{M^{i_{d}},M^{j_{d}}}.
\]

\end{lemma}

The following theorem is given in \cite{DXY4} where (1)(2) are proved
\cite[Theorem 5.4]{DXY4} and (3)(4) are proved in \cite[Theorem 6.8]{DRX2}.
We would like to highlight that corrections have been made for the
results outlined in (1)(iv) and (2)(iv), along with revisions to their
respective proofs.

\begin{theorem} \label{S-matrix in permutation orbifold}Let $1\le r,s<k$,
$d=\text{gcd}\left(s,k\right),$ $d_{1}=\text{gcd}\left(r,k\right)$
and $f=\text{gcd}$$\left(d,r\right)$. Set $l=\frac{k}{d},$ $l_{1}=\frac{k}{d_{1}}$,
$c=\frac{d}{f}$, $a=\frac{d_{1}}{f}.$ Suppose $i_{1},\cdots,i_{f},j_{1},\cdots,j_{f}\in\left\{ 0, 1, \cdots,p\right\} ,$
$j_{1}\le\cdots\le j_{f}$.

% Preview source code from paragraph 152 to 153

\begin{enumerate}[label=(\arabic*)]

\item If $|\{i_{1},\cdots,i_{f}\}|>1$ and $i_{1}\le\cdots\le i_{f}$,
then for $0\le t<l_{1}$, $0\le t_{1}<l$, $0\le n<k$,

\[
S_{\left(T_{g^{r}}^{i_{1},\cdots,i_{f};a}\right)^{t},N}=\frac{1}{l_{1}}\begin{cases}
e^{\frac{2\pi i\left(st+rt_{1}\right)}{k}}\sum_{i=0}^{d-1}S_{T_{g^{r}}^{i_{1},\cdots,i_{f};a},T_{g^{s}}^{j_{1},\cdots,j_{f};c}\circ g^{i}},\\
\ \ \ \ \ \text{if}\ N=\left(T_{g^{s}}^{j_{1},\cdots,j_{f};c}\right)^{t_{1}},\left|\left\{ j_{1},\cdots,j_{f}\right\} \right|>1\\
\ \ \text{\ \ \ and}\ \ d_{1}m\equiv s\ (\text{mod}\ k)\ \text{for}\ \text{some}\ m;\ \ \ \ \text{\textbf{(i)}}\\
e^{\frac{2\pi i\left(st+rn\right)}{k}}S_{T_{g^{r}}^{i_{1},\cdots,i_{f};a},T_{g^{s}}^{j_{1},\cdots,j_{f};b},}\\
\ \ \ \text{\ \ if}\ N=\left(T_{g^{s}}^{j_{1},\cdots,j_{f};c}\right)^{n},\left|\left\{ j_{1},\cdots,j_{f}\right\} \right|=1\\
\ \ \text{\ \ \ and}\ d_{1}m\equiv s\ (\text{mod}\ k)\ \text{for}\ \text{some}\ m;\ \ \ \ \boldsymbol{\text{\textbf{(ii)}}}\\
\sum_{i=0}^{r-1}S_{T_{g^{r}}^{i_{1},\cdots,i_{f};a},\left(\left(M^{j_{1}}\right)^{\otimes l_{1}}\otimes\cdots\otimes\left(M^{j_{d_{1}}}\right)^{\otimes l_{1}}\right)\circ g^{i}},\\
\ \ \ \text{\ \ if}\ N=\left(M^{j_{1}}\right)^{\otimes l_{1}}\otimes\cdots\otimes\left(M^{j_{d_{1}}}\right)^{\otimes l_{1}}\text{and\ \ }d_{1}>1;\ \ \ \ \text{\textbf{(iii)}}\\
e^{\frac{2\pi irn}{k}}S_{T_{g^{r}}^{i_{1},\cdots,i_{f};a},M^{j,\cdots,j}},\quad 
\text{ \text{\ \ if}}\ N=\left(M^{j,\cdots,j}\right)^{n};\ \ \ \ \text{\textbf{(iv)}}\\
0,\ \text{\ \ otherwise}\ \ \ \ \text{\textbf{(v)}}
\end{cases}
\]
where $S_{T_{g^{r}}^{i_{1},\cdots,i_{f};a},T_{g^{s}}^{j_{1},\cdots,j_{f};c}}$,
$S_{T_{g^{r}}^{i_{1},\cdots,i_{f};a},\left(M^{j_{1}}\right)^{\otimes l_{1}}\otimes\cdots\otimes\left(M^{j_{d_{1}}}\right)^{\otimes l_{1}}}$
and $S_{T_{g^{r}}^{M^{i_{1}}},M^{j,\cdots,j}}$ are given in \cite[Lemmas 4.4 and 4.5]{DXY4}.

\item If $|\{i_{1},\cdots,i_{f}\}|=1$, then for $0\le t<k,$ $0\le t_{1}<l$,
$0\le n<k$, 
\[
S_{\left(T_{g^{r}}^{i_{1},\cdots,i_{f};a}\right)^{t},N}=\frac{1}{k}\begin{cases}
e^{\frac{2\pi i\left(st+rt_{1}\right)}{k}}\sum_{i=0}^{d-1}S_{T_{g^{r}}^{i_{1},\cdots,i_{f};a},T_{g^{s}}^{j_{1},\cdots,j_{f};c}\circ g^{i}},\\
\ \ \ \text{\ \ if}\ N=\left(T_{g^{s}}^{j_{1},\cdots,j_{f};c}\right)^{t_{1}},\left|\left\{ j_{1},\cdots,j_{f}\right\} \right|>1;\ \ \ \ \boldsymbol{\boldsymbol{\text{\textbf{(i)}}}}\\
e^{\frac{2\pi i\left(st+rn\right)}{k}}S_{T_{g^{r}}^{i_{1},\cdots,i_{f};a},T_{g^{s}}^{j_{1},\cdots,j_{f};c},}\\
\ \ \ \ \ \text{if}\ N=\left(T_{g^{s}}^{j_{1},\cdots,j_{f};c}\right)^{n},\left|\left\{ j_{1},\cdots,j_{f}\right\} \right|=1;\ \ \ \ \text{\textbf{(ii)}}\\
\sum_{i=0}^{r-1}S_{T_{g^{r}}^{i_{1},\cdots,i_{f};a},\left(\left(M^{j_{1}}\right)^{\otimes l_{1}}\otimes\cdots\otimes\left(M^{j_{d_{1}}}\right)^{\otimes l_{1}}\right)\circ g^{i}},\\
\ \ \ \text{\ \ if}\ N=\left(M^{j_{1}}\right)^{\otimes l_{1}}\otimes\cdots\otimes\left(M^{j_{d_{1}}}\right)^{\otimes l_{1}}\ \text{and\ \ }d_{1}>1;\ \ \ \ \text{\textbf{(iii)}}\\
e^{\frac{2\pi irn}{k}}S_{T_{g^{r}}^{i_{1},\cdots,i_{f};a},M^{j,\cdots,j}},\quad
 \text{\ \ if}\ N=\left(M^{j,\cdots,j}\right)^{n};\ \ \ \ \text{\textbf{(iv)}}\\
0,\ \ \text{otherwise}\ \ \ \ \text{\textbf{(v)}} & ,
\end{cases}
\]
where $S_{T_{g^{r}}^{i_{1},\cdots,i_{f};a},T_{g^{s}}^{j_{1},\cdots,j_{f};c}}$,
$S_{T_{g^{r}}^{i_{1},\cdots,i_{f};a},\left(M^{j_{1}}\right)^{\otimes l_{1}}\otimes\cdots\otimes\left(M^{j_{d_{1}}}\right)^{\otimes l_{1}}}$
and $S_{T_{g^{r}}^{i_{1},\cdots,i_{f};a},M^{j,\cdots,j}}$ are given
in \cite[Lemmas 4.4 and 4.5]{DXY4}.

\item Let $A=\{0,1, \cdots, p\}$ and $I$ be a subset of $A^k\setminus\{(i, \cdots, i)\mid i\in A\}$ consisting of the orbit representatives under the action of $\langle g\rangle$. Then for  $\left(i_{1},\cdots,i_{k}\right),$$\left(j_{1},\cdots,j_{k}\right)\in I$,
$j\in\left\{ 0, 1, \cdots,p\right\} $ and $0\le n<k$, one has
\[
S_{M^{i_{1},\cdots,i_{k}},N}=\begin{cases}
\sum_{n=0}^{k-1}\prod_{t=1}^{k}S_{i_{t},j_{t}+n}, & \ \text{if}\ N=M^{j_{1},\cdots,j_{k}};\\
\prod_{t=1}^{k}S_{i_{t},j}, & \text{if}\ N=\left(M^{j,\cdots,j}\right)^{n}.
\end{cases}
\]

\item Let $i,j\in\left\{ 0,1,\cdots,p\right\} $ and $0\le m,n<k$.
We have 
\[
S_{\left(M^{i,\cdots,i}\right)^{m},\left(M^{j,\cdots,j}\right)^{n}}=\frac{1}{k}S_{i,j}^{k}.
\]

\end{enumerate} 

\end{theorem}
\begin{proof}
For the proper validation of results (1)(iv) and (2)(iv), we will
exclude the conditions $d_{1}=f=1$ from the associated results and
proofs as presented in {[}DXY, Theorem 5.4{]}. Therefore we obtain
\[
S_{\left(T_{g^{r}}^{i_{1},\cdots,i_{f};a}\right)^{t},\left(M^{j,\cdots,j}\right)^{n}}=\frac{1}{l_{1}}e^{\frac{2\pi irn}{k}}S_{T_{g^{r}}^{i_{1},\cdots,i_{f};a},M^{j,\cdots,j}}
\]
for the case (1)(iv) and 
\[
S_{\left(T_{g^{r}}^{i_{1},\cdots,i_{f};a}\right)^{t},\left(M^{j,\cdots,j}\right)^{n}}=\frac{1}{k}e^{\frac{2\pi irn}{k}}S_{T_{g^{r}}^{i_{1},\cdots,i_{f};a},M^{j,\cdots,j}}
\]
for the case (2)(iv).
\end{proof}

\subsection{Fusion product of an untwisted  and a twisted module}

The following theorems are given in \cite{DLXY}:

\begin{theorem} \label{fusion product of untwisted and twisted}Let
$V$ be any vertex operator algebra and set $g=(1, 2,\cdots, k)\in\text{Aut}(V^{\otimes k})$
with $k$ a positive integer. Let $M$ and $N$ be any $V$-modules
such that a tensor product $M\boxtimes N$ exists. Then 
\begin{align*}
\left(V^{\otimes(j-1)}\otimes M\otimes V^{\otimes(k-j)}\right)\boxtimes T_{g}(N)\simeq T_{g}(M\boxtimes N),
\end{align*}
 for $1\le j\le k$. In particular, we have $T_{g}\left(M\right)=\left(V^{\otimes(j-1)}\otimes M\otimes V^{\otimes(k-j)}\right)\boxtimes T_{g}\left(V\right)$.
\end{theorem}

\begin{theorem} \label{fusion product of untwisted and twisted-General}Let
$V$ be a regular and self-dual vertex operator algebra of CFT type,
and let $M_{1},...,M_{k},N$ be $V$-modules. Let $g$ be a $k$-cycle
permutation. Then 
\[
\left(M_{1}\otimes\cdots\otimes M_{k}\right)\boxtimes_{V^{\otimes k}}T_{g}\left(N\right)\simeq T_{g}\left(M_{1}\boxtimes_{V}\cdots\boxtimes_{V}M_{k}\boxtimes_{V}N\right).
\]
\end{theorem}

\section{Fusion Products of Twisted Modules in Permutation Orbifolds }

Let  $\sigma$ be any permutation in $S_k$. In this section, we first determine the fusion products of $\sigma^s$ and $\sigma^r$-twisted modules.

For simplicity, we denote a $\sigma^s$-twisted  $V^{\otimes k}$-module by $\mathcal{T}_{\sigma^s}$. 
We aim to find $\mathcal{T}_{\sigma^s} \boxtimes \mathcal{T}_{\sigma^r}$. If $\text{gcd}(s, r) = d' > 1$, then $s = d's'$ and $r = d'r'$ for some $ s', r' \ge 1$, with $\text{gcd}(s', r') = 1$. Assume that $\sigma^{d'} = \sigma_1 \cdots \sigma_s$ is a product of disjoint cycles. Then $\sigma^s = \sigma_1^{s'} \cdots \sigma_s^{s'}$ and $\sigma^r = \sigma_1^{r'} \cdots \sigma_s^{r'}$. Thus,
\[
\mathcal{T}_{\sigma^s} \boxtimes \mathcal{T}_{\sigma^r} = \left( \mathcal{T}_{\sigma_1^{s'}} \otimes \cdots \otimes \mathcal{T}_{\sigma_s^{s'}} \right) \boxtimes_{V^{\otimes k}} \left( \mathcal{T}_{\sigma_1^{r'}} \otimes \cdots \otimes \mathcal{T}_{\sigma_s^{r'}} \right),
\]
\[
= \left( \mathcal{T}_{\sigma_1^{s'}} \boxtimes \mathcal{T}_{\sigma_1^{r'}} \right) \otimes \cdots \otimes \left( \mathcal{T}_{\sigma_s^{s'}} \boxtimes \mathcal{T}_{\sigma_s^{r'}} \right).
\]
Therefore, it suffices to find the fusion products $\mathcal{T}_{\sigma_i^s} \boxtimes \mathcal{T}_{\sigma_i^r}$, $1\le i\le s$, where $\sigma_i$ is a cycle and $\text{gcd}(s, r) = 1$.  Consequently, in the following we will focus on  the cyclic permutation orbifold $\left(V^{\otimes k}\right)^{\left\langle g \right\rangle }$ where $g=(1, 2, \cdots, k)$. We will derive an explicit formula for the fusion product of an irreducible $g^s$-twisted $V^{\otimes k}$-module with an irreducible $g^r$-twisted $V^{\otimes k}$-module, where  $1 \leq s, r < k$ and $\text{gcd}(s, r) = 1$. 

Recall that $\text{Irr}(V) = \{ V = M^0, M^1, \dots, M^p \}$ denotes the set of all irreducible $V$-modules up to isomorphism.
 Assume that $\text{gcd}(s, k) = d$. Then $o(g^s) = \frac{k}{d}$, and $g^s = \tau_1 \cdots \tau_d$, where $\tau_i$ are $\frac{k}{d}$-cycles for $1 \leq i \leq d$. By \cite{BDM}, each irreducible $g^s$-twisted module can be expressed as
\[
T_{\tau_1, \dots, \tau_d}^{i_1, \dots, i_d} = T_{\tau_1}(M^{i_1}) \otimes \cdots \otimes T_{\tau_d}(M^{i_d}),
\]
where $i_1, \dots, i_d \in \{ 0, 1, \dots, p \}$. In particular,
\[
T_{\tau_1, \dots, \tau_d}^{0, \dots, 0} = T_{\tau_1}(V) \otimes \cdots \otimes T_{\tau_d}(V).
\]
For simplicity, we will sometimes denote the module $T_{\tau_1, \dots, \tau_d}^{i_1, \dots, i_d}$ by $T_{g^s}^{i_1, \dots, i_d}$. Similarly, any $g^r$-twisted $V^{\otimes k}$-module can be expressed as $T_{g^r}^{j_1, \dots, j_m} = T_{\sigma_1, \dots, \sigma_m}^{j_1, \dots, j_m} = T_{\sigma_1}(M^{j_1}) \otimes \cdots \otimes T_{\sigma_m}(M^{j_m})$, where $m = \text{gcd}(r, k)$, $\sigma_1, \dots, \sigma_m$ are $\frac{k}{m}$-cycles, and $j_1, \dots, j_m \in \{ 0, 1, \dots, p \}$.

By Theorem \ref{fusion product of untwisted and twisted} and the associativity of the fusion product \cite{DLXY}, we first determine the fusion product of the $g^s$-twisted module $T_{g^s}^{0, \dots, 0}$ and the $g^r$-twisted module $T_{g^r}^{0, \dots, 0}$. We then apply this result to obtain the formula for the fusion product of the $g^s$-twisted module $T_{g^s}^{i_1, \dots, i_d}$ and the $g^r$-twisted module $T_{g^r}^{j_1, \dots, j_m}$.

\subsection{\label{subsec: fusion step-1}The fusion product of the $g^{s}$-twisted
module $T_{g^{s}}^{0,\cdots,0}$ and the $g^{r}$-twisted module $T_{g^{r}}^{0,\cdots,0}$}

Let $g^{s+r} = \gamma_1 \cdots \gamma_q$ be a product of $q$ disjoint cycles, where $q = \text{gcd}(s+r, k)$ and each $\gamma_i$ is a $\frac{k}{q}$-cycle for $1 \leq i \leq q$. Then any $g^{s+r}$-twisted module can be written as 
\[
T_{g^{s+r}}^{t_1, \dots, t_q} = T_{\gamma_1, \dots, \gamma_q}^{t_1, \dots, t_q} = T_{\gamma_1}(M^{t_1}) \otimes \cdots \otimes T_{\gamma_q}(M^{t_q}),
\]
with $t_1, \dots, t_q \in \{ 0, 1, \dots, p \}$.

The following lemma lists some S-matrix entries for the orbifold algebra, taken from Lemma \ref{S-matrix, nontwisted and twisted} and Theorem \ref{S-matrix in permutation orbifold}, for later use.

\begin{lemma} \label{S-matrix on perm orb from DXY} Let $g=(1, 2,\cdots, k)$ be a $k$-cycle, and let $1 \leq s, r < k$. Define $g^s = \tau_1  \cdots \tau_d$, $g^r = \sigma_1 \cdots \sigma_m$, and $g^{s+r} = \gamma_1 \cdots \gamma_q$ as products of disjoint cycles, where $d = \text{gcd}(s, k)$, $m = \text{gcd}(r, k)$, and $q = \text{gcd}(s+r, k)$. Denote $M^{j, \dots, j} = (M^j)^{\otimes k}$ for $j \in \{ 0, 1, \dots, p \}$.

Then, for $0 \leq t, n < k$ and $i, j, t_1, \dots, t_n \in \{ 0, 1, \dots, p \}$, we have the following $S$-matrix entries:

\begin{equation}
S_{N^t, \left(M^{j, \dots, j}\right)^n} = 
\begin{cases}
\frac{1}{k} e^{\frac{2\pi i s n}{k}} S_{0, j}^d, & \text{if } N = T_{\tau_1, \cdots, \tau_d}^{0, \cdots, 0}; \\
\frac{1}{k} e^{\frac{2\pi i r n}{k}} S_{0, j}^m, & \text{if } N = T_{\sigma_1, \cdots, \sigma_m}^{0, \cdots, 0}; \\
\frac{q}{k} e^{\frac{2\pi i (s+r) n}{k}} S_{t_1, j} \cdots S_{t_q, j}, & \text{if } N = T_{\gamma_1, \cdots, \gamma_q}^{t_1, \cdots, t_q} \text{ and } q > 1; \\
\frac{1}{k} S_{i, j}^k, & \text{if } N = M^{i, \dots, i}.
\end{cases}
\label{S-matrix entries twisted-1 and untwisted}
\end{equation}

\end{lemma}
Let $G = \langle g \rangle$ be the cyclic group generated by the $k$-cycle $g=(1, 2,\cdots, k)$. Then the category $\mathcal{C}_{(V^{\otimes k})^G}$ of ordinary $(V^{\otimes k})^G$-modules is a modular tensor category, and $V^{\otimes k}$ is a regular commutative algebra in $\mathcal{C}_{(V^{\otimes k})^G}$. Denote the isomorphism classes of simple objects of $\mathcal{C}_{(V^{\otimes k})^G}$ by $$\mathbf{O}(\mathcal{C}_{(V^{\otimes k})^G}) = \{\lambda \mid \lambda \in I\}.$$

Let $\text{Rep}(V^{\otimes k})$ be the module category of the algebra $V^{\otimes k}$ in the category $\mathcal{C}_{(V^{\otimes k})^G}$ as defined in Definition \ref{Rep(V)-def}. Then $\text{Rep}(V^{\otimes k})$ is a fusion category, and all simple objects in $\text{Rep}(V^{\otimes k})$ are $T_{\tau_1, \dots, \tau_d}^{i_1, \dots, i_d}$, where $i_1, \dots, i_d \in \{ 0, 1, \dots, p \}$ and $\tau_s$ are $\frac{k}{d}$-cycles for $1 \leq s \leq d$. In particular, if $\tau_1 = \cdots = \tau_d = 1$, then $d = k$ and $T_{\tau_1, \dots, \tau_d}^{i_1, \dots, i_d} = M^{i_1} \otimes \cdots \otimes M^{i_k}$ is an untwisted $V^{\otimes k}$-module.

The fusion ring $K(\text{Rep}(V^{\otimes k}))$ is a semisimple associative algebra by Theorem \ref{Lustig fusion algebra}. Denote the center of the fusion ring $K(\text{Rep}(V^{\otimes k}))$ by $Z(K(\text{Rep}(V^{\otimes k})))$. Let $\text{Rep}(V^{\otimes k})^0$ be the local $V^{\otimes k}$-module category. Then $\text{Rep}(V^{\otimes k})^0$ is a braided fusion category. Moreover, since $\mathcal{C}_{(V^{\otimes k})^G}$ is a modular tensor category, so is $\text{Rep}(V^{\otimes k})^0$ \cite{KO}. Denote the isomorphism classes of simple objects of $\text{Rep}(V^{\otimes k})^0$ by $$\mathbf{O}(\text{Rep}(V^{\otimes k})^0) = \{\sigma_i \mid i \in \Delta \}.$$

Now we review results from \cite[Theorem 7.1]{DRX3}. Consider the modular tensor category $\mathcal{C}_{(V^{\otimes k})^G}$ and the regular commutative algebra $V^{\otimes k} \in \mathcal{C}_{(V^{\otimes k})^G}$. For any $M \in \text{Rep}(V^{\otimes k})$, define a linear map $$T_M: K(\text{Rep}(V^{\otimes k})) \to K(\text{Rep}(V^{\otimes k}))$$ such that $T_M(N) = M \boxtimes_{V^{\otimes k}} N$ for $N \in \mathbf{O}(\text{Rep}(V^{\otimes k}))$, where $K(\text{Rep}(V^{\otimes k}))$ is the fusion algebra of the fusion category $\text{Rep}(V^{\otimes k})$. Set $$a_\lambda = (V^{\otimes k}) \boxtimes_{(V^{\otimes k})^G} \lambda$$ for $\lambda \in \mathbf{O}(\mathcal{C}_{(V^{\otimes k})^G}) = \{\lambda \mid \lambda \in I\}$. 
Then $a_\lambda \in Z(K(\text{Rep}(V^{\otimes k})))$, and $K(\text{Rep}(V^{\otimes k}))$ has a basis consisting of common eigenvalues $v^{(i, \mu, m)}$, where $i \in \{0, 1, \dots, p\}$ and $\mu \in I$.

To distinguish $S$-matrices for different vertex operator algebras, we will use $(S_{\lambda, \mu})$ for $\lambda, \mu \in I$ to denote the $S$-matrix of the orbifold vertex operator algebra $(V^{\otimes k})^G$. We will continue to use $(S_{i,j})_{i,j=0}^{p}$ as defined earlier to denote the $S$-matrix for the vertex operator algebra $V$.

Define $T_{a_\lambda} = T_\lambda$ and $T_{\sigma_i} = T_i$. The operators $T_j$ and $T_\lambda$, for $j \in \Delta$ and $\lambda \in I$, can be diagonalized simultaneously. Thus, there exists an orthonormal basis $v^{(i, \mu, m)}$ with $i \in \Delta$ and $\mu \in I$, such that
\[
T_j v^{(i, \mu, m)} = \frac{S_{j,i}}{S_{0,i}} v^{(i, \mu, m)}, \quad T_\lambda v^{(i, \mu, m)} = \frac{S_{\lambda,m}}{S_{0,\mu}} v^{(i, \mu, m)},
\]
for all $j$ and $\lambda$, where $m$ indexes the basis vectors in the eigenspace corresponding to the indicated eigenvalues.

For simplicity, we will use $\sum_i$ to denote $\sum_{i=0}^{p}$ in the following discussion. The next lemma is essential for the upcoming theorem.

\begin{lemma} \label{identity condition }
Let $g = (1, 2, \dots, k)$ be a $k$-cycle, and let $1 \leq s, r < k$. Suppose $g^s = \tau_1 \cdots \tau_d$, $g^r = \sigma_1 \cdots \sigma_m$, and $g^{s+r} = \gamma_1 \cdots \gamma_q$, where each expression represents a product of disjoint cycles. Here, $d = \gcd(s, k)$, $m = \gcd(r, k)$, and $q = \gcd(s + r, k)$. Assume that
\[
T_{\tau_1, \dots, \tau_d}^{0, \dots, 0} \boxtimes_{V^{\otimes k}} T_{\sigma_1, \dots, \sigma_m}^{0, \dots, 0} = \sum_{t_1, \dots, t_q} n_{t_1, \dots, t_q} T_{\gamma_1, \dots, \gamma_q}^{t_1, \dots, t_q},
\]
then the coefficients $n_{t_1, \dots, t_q}$ satisfy the following identity:
\[
\frac{1}{S_{0,j}^{k-d-m}} = \sum_{t_1, \dots, t_q} n_{t_1, \dots, t_q} S_{t_1,j} \cdots S_{t_q,j}.
\]

\end{lemma}

\begin{proof}

For the \( g^s \)-twisted module \( T_{g^s}^{0, \dots, 0}=T_{\tau_1, \dots, \tau_d}^{0, \dots, 0}\), we have \( G_{T_{g^s}^{0, \dots, 0}} = \langle g \rangle \), so \( |G_{T_{g^s}^{0, \dots, 0}}| = k \). The cardinality of the \( G \)-orbit \( T_{g^s}^{0, \dots, 0} \circ G \) is \( [G : G_{T_{g^s}^{0, \dots, 0}}] = 1 \). Similarly, the cardinality of the \( G \)-orbit of the \( g^r \)-twisted module \( T_{g^r}^{0, \dots, 0} \) is also 1. For the \( g^{s+r} \)-twisted module \( T_{g^{s+r}}^{t_1, \dots, t_q} \), we have \( G_{T_{g^{s+r}}^{t_1, \dots, t_q}} = \langle g^q \rangle \), so \( |G_{T_{g^{s+r}}^{t_1, \dots, t_q}}| = \frac{k}{q} \), and the cardinality of the \( G \)-orbit \( T_{g^{s+r}}^{t_1, \dots, t_q} \circ G \) is \( [G : G_{T_{g^{s+r}}^{t_1, \dots, t_q}}] = q \).

Let \( \lambda_1 = (T_{g^s}^{0, \dots, 0})^t \), \( \lambda_2 = (T_{g^r}^{0, \dots, 0})^t \), \( \lambda_3 = (T_{g^{s+r}}^{t_1, \dots, t_q})^a \), and \( \mu = (M^{j, \dots, j})^n \) with \( 0 \leq t, n < k \) and \( 0 \leq a < q \). By \cite[(6.3)]{DNR}, we have:
$$
a_{\lambda_1} = V^{\otimes k} \boxtimes_{(V^{\otimes k})^G} (T_{g^s}^{0, \dots, 0})^t = \sum_{i=0}^{k-1} (V^{\otimes k})^i \boxtimes_{(V^{\otimes k})^G} (T_{g^s}^{0, \dots, 0})^t = T_{g^s}^{0, \dots, 0}.
$$
Similarly, \( a_{\lambda_2} = T_{g^r}^{0, \dots, 0} \). For \( \lambda_3 = (T_{g^{s+r}}^{t_1, \dots, t_q})^a \), we have:
$$
a_{\lambda_3} = V^{\otimes k} \boxtimes_{(V^{\otimes k})^G} (T_{g^{s+r}}^{t_1, \dots, t_q})^a = \bigoplus_{j=0}^{q-1} T_{g^{s+r}}^{t_1, \dots, t_q} \circ g^j = \bigoplus_{j=0}^{q-1} T_{g^{s+r}}^{\sigma^j(t_1), \dots, \sigma^j(t_q)}, \quad 0 \leq a < q,
$$
where \( \sigma = (t_1, t_q, t_{q-1}, \dots, t_2) \) (see \cite[3.2.1]{DXY4}). Additionally,
$$
a_\mu = V^{\otimes k} \boxtimes_{(V^{\otimes k})^G} (M^{j, \dots, j})^n = \bigoplus_{t=0}^{k-1} (V^{\otimes k})^t \boxtimes_{(V^{\otimes k})^G} (M^{j, \dots, j})^n = M^{j, \dots, j}.
$$

By Lemma \ref{S-matrix on perm orb from DXY}, we obtain:
\[
\frac{S_{\lambda_1, \mu}}{S_{0, \mu}} = \frac{S_{(T_{g^s}^{0, \dots, 0})^t, (M^{j, \dots, j})^n}}{S_{(V^{\otimes k})^{\langle g \rangle}, (M^{j, \dots, j})^n}} = \frac{\frac{1}{k} e^{\frac{2\pi isn}{k}} S_{0,j}^d}{\frac{1}{k} S_{0,j}^k} = e^{\frac{2\pi isn}{k}} S_{0,j}^{d-k},
\]
\[
\frac{S_{\lambda_2, \mu}}{S_{0, \mu}} = \frac{S_{(T_{g^r}^{0, \dots, 0})^t, (M^{j, \dots, j})^n}}{S_{(V^{\otimes k})^{\langle g \rangle}, (M^{j, \dots, j})^n}} = \frac{\frac{1}{k} e^{\frac{2\pi irn}{k}} S_{0,j}^m}{\frac{1}{k} S_{0,j}^k} = e^{\frac{2\pi irn}{k}} S_{0,j}^{m-k},
\]
and
\[
\frac{S_{\lambda_3, \mu}}{S_{0, \mu}} = \frac{S_{(T_{g^{s+r}}^{t_1, \dots, t_q})^q, (M^{j, \dots, j})^n}}{S_{(V^{\otimes k})^{\langle g \rangle}, (M^{j, \dots, j})^n}} = \frac{\frac{q}{k} e^{\frac{2\pi i(s+r)n}{k}} S_{t_1,j} \cdots S_{t_q,j}}{\frac{1}{k} S_{0,j}^k} = q e^{\frac{2\pi i(s+r)n}{k}} \frac{S_{t_1,j} \cdots S_{t_q,j}}{S_{0,j}^k}.
\]

Since 
\[
a_{\lambda_1} \boxtimes_{V^{\otimes k}} a_{\lambda_2} = a_{\lambda_1 \boxtimes_{(V^{\otimes k})^G} \lambda_2} = \sum_{\lambda_3} N_{\lambda_1, \lambda_2}^{\lambda_3} a_{\lambda_3},
\]
by \cite[Theorem 7.1]{DRX3}, we have:
\[
\frac{S_{\lambda_1, \mu}}{S_{0, \mu}} \cdot \frac{S_{\lambda_2, \mu}}{S_{0, \mu}} = \sum_{\lambda_3} N_{\lambda_1, \lambda_2}^{\lambda_3} \frac{S_{\lambda_3, \mu}}{S_{0, \mu}}.
\]
Denote
\[
\mathcal{P} = \{(t_1, \dots, t_q) \mid N_{\lambda_1, \lambda_2}^{\lambda} \neq 0 \text{ where } \lambda = (T_{g^{s+r}}^{t_1, \dots, t_q})^\ell, \, t_1, \dots, t_q \in \{0,1,\dots,p\} \}.
\]
Then,
\[
 e^{\frac{2\pi isn}{k}} S_{0,j}^{d-k} \cdot  e^{\frac{2\pi irn}{k}} S_{0,j}^{m-k} = \sum_{( t_1, \dots, t_q) \in \mathcal{P}} N_{\lambda_1, \lambda_2}^{(T_{g^{s+r}}^{t_1, \dots, t_q})^\ell} q e^{\frac{2\pi i(s+r)n}{k}} \frac{S_{t_1,j} \cdots S_{t_q,j}}{S_{0,j}^k},
\]
and hence
\begin{equation}
S_{0,j}^{d+m-k} = \sum_{(t_1, \dots, t_q) \in \mathcal{P}} N_{\lambda_1, \lambda_2}^{(T_{g^{s+r}}^{t_1, \dots, t_q})^\ell} q S_{t_1,j} \cdots S_{t_q,j}. \label{eq1}
\end{equation}

Furthermore, since \( a_{\lambda_3} = \bigoplus_{j=0}^{q-1} T_{g^{s+r}}^{\sigma^j(t_1), \dots, \sigma^j(t_q)} \) with \( \sigma = (t_1, t_q, t_{q-1}, \dots, t_2) \), we obtain:
\[
\sum_{\lambda_3} N_{\lambda_1, \lambda_2}^{\lambda_3} a_{\lambda_3} = \sum_{(t_1, \dots, t_q) \in \mathcal{P}} \bigoplus_{j=0}^{q-1} n_{\sigma^j(t_1), \dots, \sigma^j(t_q)} T_{g^{s+r}}^{\sigma^j(t_1), \dots, \sigma^j(t_q)}.
\]
By Theorem \ref{S-matrix in permutation orbifold}, \( n_{t_1, \dots, t_q} = n_{\sigma(t_1), \dots, \sigma(t_q)} = \cdots = n_{\sigma^{q-1}(t_1), \dots, \sigma^{q-1}(t_q)} \). Thus, we have:
\[
\sum_{(t_1, \dots, t_q) \in \mathcal{P}} \bigoplus_{j=0}^{q-1} n_{\sigma^j(t_1), \dots, \sigma^j(t_q)} T_{g^{s+r}}^{\sigma^j(t_1), \dots, \sigma^j(t_q)} = \sum_{(t_1, \dots, t_q) \in \mathcal{P}} n_{t_1, \dots, t_q} \bigoplus_{j=0}^{q-1} T_{g^{s+r}}^{\sigma^j(t_1), \dots, \sigma^j(t_q)}.
\]
Using (\ref{eq1}), we obtain:
\[
S_{0,j}^{d+m-k} = \sum_{(t_1, \dots, t_q) \in \mathcal{P}} N_{\lambda_1, \lambda_2}^{(T_{g^{s+r}}^{t_1, \dots, t_q})^{\ell}} q S_{t_1,j} \cdots S_{t_q,j} = \sum_{t_1, \dots, t_q} n_{t_1, \dots, t_q} S_{t_1,j} \cdots S_{t_q,j}.
\]
Thus,
\[
\frac{1}{S_{0,j}^{k-d-m}} = \sum_{t_1, \dots, t_q} n_{t_1, \dots, t_q} S_{t_1,j} \cdots S_{t_q,j}.
\]
\end{proof}

We present some straightforward calculations that will be utilized in our proof.

\begin{lemma}\label{calculation of fusion and S-matrix} For $t,j,a,b,c\in\left\{ 0,1,\cdots,p\right\} ,$we
have

\begin{enumerate}[label=(\arabic*)]

\item $\sum_{a,b}N_{b',c}^{a}S_{j,a}S_{j,b}=\frac{S_{c,j}}{S_{0,j}}$
;

\item $\sum_{a}N_{b,c}^{a}S_{a,j}=\frac{S_{b,j}S_{c,j}}{S_{0,j}}$; 

\item $\sum_{c}N_{b,c,}^{a}S_{c',j}=\frac{S_{b,j}S_{a',j}}{S_{0,j}};$

\item $\sum_{a}N_{b,c}^{a}N_{a,f}^{e}=\sum_{t}\frac{S_{b,t}S_{c,t}S_{f,t}S_{e',t}}{S_{0,t}^{2}}$; 

\item For any $i_{1},\cdots i_{n},l_{1},\cdots,l_{n-1},j\in\left\{ 0,1,\cdots,p\right\} $
and $n\in\mathbb{N}$, $n\ge2,$ 
\[
\sum_{l_1, l_2,\cdots, l_{n-1}} N_{i_{1},i_{2}}^{l_{1}}N_{l_{1},i_{3}}^{l_{2}}\cdots N_{l_{n-2},i_{n}}^{l_{n-1}}S_{l_{n-1},j}=\frac{S_{i_{1},j}S_{i_{2},j}\cdots S_{i_{n},j}}{S_{0,j}^{n-1}}.
\]

\end{enumerate}

\end{lemma}

\begin{proof}

\begin{enumerate}[label=(\arabic*)]

\item By Theorem \ref{Verlinde formula} and straightforward
calculation, we have 
\[
\sum_{a,b}N_{b',c}^{a}S_{j,a}S_{j,b}=\sum_{a,b,t}\frac{S_{b',t}S_{c,t}S_{a,t}^{-1}}{S_{0,t}}S_{j,a}S_{j,b}=\sum_{b}\frac{S_{b',j}S_{c,j}}{S_{0,j}}S_{j,b}=\frac{S_{c,j}}{S_{0,j}}.
\]
\item By similar calculation as in (1), one can obtain the result.

\item Using Theorem \ref{Verlinde formula} and (2) we have 
\[
\sum_{c}N_{b,c,}^{a}S_{c',j}=\sum_{c}N_{b,a'}^{c'}S_{c',j}=\frac{S_{b,j}S_{a',j}}{S_{0,j}}.
\]
\item  By Theorem \ref{Verlinde formula} and straightforward calculation, it is easy to prove it. 
\item It is easy to prove the result by applying (2). 
\end{enumerate}
\end{proof}

Let \( g = (1, 2,  \dots, k) \) be a \( k \)-cycle, and let \( 1 \leq s, r < k \) with  \( \text{gcd}(s, r) = 1 \). Define \( A = \langle g^s \rangle \) and \( B = \langle g^r \rangle \).

\begin{lemma} \label{inter of orbits nonempty}
For any \( i,  j \in \{1, 2, \dots, k\} \), we have \( A(i) \cap B(j) \neq \emptyset \), where \( A(i) = \{ a(i) \mid a \in A \} \) is the orbit of \( i \) under the action of \( A \).
\end{lemma}

\begin{proof}
Without loss of generality, assume \( i = 1 \). Since \( \text{gcd}(s, r) = 1 \), there exist integers \( t \) and \( n \) such that \( st + rn= 1 \). Thus, we have \( g^{st} g^{rn} = g \), or equivalently, \( g^{st} = g^{-rn}g \).

Now, observe that
\[
g^{(j-1)st}(1) = g^{-(j-1)rn} g^{j-1}(1) = g^{-(j-1)rn}(j),
\]
which lies in \( A(1) \cap B(j) \), completing the proof.
\end{proof}

Now, let \( g^s = \tau_1 \cdots \tau_d \), \( g^r = \sigma_1 \cdots \sigma_m \), and \( g^{s+r} = \gamma_1 \cdots \gamma_q \) be products of disjoint cycles, where \( d = \text{gcd}(s, k) \), \( m = \text{gcd}(r, k) \), and \( q = \text{gcd}(s+r, k) \). Set \( a = k/q \). By \cite[Theorem 6.3, Lemma 6.4]{DLXY}, we have
\[
T_{\gamma_j}(M^{t_j}) = (M^{t_j} \otimes V^{\otimes(a-1)}) \boxtimes T_{\gamma_j}(V)
\]
for \( 1 \leq j \leq q \). Thus,
\begin{align}
T_{g^{s+r}}^{t_1, \dots, t_q} &= T_{\gamma_1}(M^{t_1}) \otimes \cdots \otimes T_{\gamma_q}(M^{t_q}) \nonumber \\
&= \left( (M^{t_1} \otimes V^{\otimes(a-1)}) \boxtimes T_{\gamma_1}(V) \right) \otimes \cdots \otimes \left( (M^{t_q} \otimes V^{\otimes(a-1)}) \boxtimes T_{\gamma_q}(V) \right) \nonumber \\
&= \left( (M^{t_1} \otimes V^{\otimes(a-1)}) \otimes \cdots \otimes (M^{t_q} \otimes V^{\otimes(a-1)}) \right) \boxtimes \left( T_{\gamma_1}(V) \otimes \cdots \otimes T_{\gamma_q}(V) \right) \nonumber \\
&= \left( M^{t_1} \otimes V^{\otimes(k-1)} \right) \boxtimes \cdots \boxtimes \left( V^{\otimes(q-1)} \otimes M^{t_q} \otimes V^{\otimes(k-q)} \right) \boxtimes T_{g^{s+r}}^{0, \dots, 0}. \label{g^(s+r)}
\end{align}

Let \( N \) be a \( V \)-module and \( 1 \leq i \leq k \). For convenience, we denote the \( V^{\otimes k} \)-module \( V^{\otimes(i-1)} \otimes N \otimes V^{\otimes(k-i)} \) by \( N^i \).

\begin{lemma} \label{move to 1st  factor} For any  $V$-module $N$ and any $1\leq i\leq k,$
we have 
$$N^i\boxtimes T_{g^{s}}^{0,\cdots,0}\boxtimes_{V^{\otimes k}}T_{g^{r}}^{0,\cdots,0}=
N^1\boxtimes T_{g^{s}}^{0,\cdots,0}\boxtimes_{V^{\otimes k}}T_{g^{r}}^{0,\cdots,0}.$$
\end{lemma}
\begin{proof}
Note that 
$$N^i\boxtimes_{V^{\otimes k}} T_{g^{s}}^{0,\cdots,0}\boxtimes_{V^{\otimes k}}T_{g^{r}}^{0,\cdots,0}=
 T_{g^{s}}^{0,\cdots,0}\boxtimes_{V^{\otimes k}} N^i\boxtimes_{V^{\otimes k}}T_{g^{r}}^{0,\cdots,0}.$$
By Lemma \ref{inter of orbits nonempty}, $A(i)\cap B(1)\ne \emptyset.$ Let $j\in A(i)\cap B(1).$  Then 
$$N^i\boxtimes_{V^{\otimes k}}T_{g^{r}}^{0,\cdots,0}=N^j\boxtimes_{V^{\otimes k}}T_{g^{r}}^{0,\cdots,0}$$
and
$$N^i\boxtimes_{V^{\otimes k}} T_{g^{s}}^{0,\cdots,0}\boxtimes_{V^{\otimes k}}T_{g^{r}}^{0,\cdots,0}=
 T_{g^{s}}^{0,\cdots,0}\boxtimes_{V^{\otimes k}} N^j\boxtimes_{V^{\otimes k}}T_{g^{r}}^{0,\cdots,0}=N^j\boxtimes_{V^{\otimes k}} T_{g^{s}}^{0,\cdots,0}\boxtimes_{V^{\otimes k}}T_{g^{r}}^{0,\cdots,0}.$$
Since $1,j$ are in the same $B$-orbit, we see that 
$N^j\boxtimes_{V^{\otimes k}} T_{g^{s}}^{0,\cdots,0}=N^1\boxtimes_{V^{\otimes k}} T_{g^{s}}^{0,\cdots,0}.$ The result follows.
\end{proof}

In the following we will denote the $V^{\otimes k}$-module $M^{a}\otimes V^{\left(k-1\right)}$
by $M^{a,0,\cdots,0}$ for simplicity.

\begin{theorem} \label{T_g(V) and T_g^s(0,..0)}Let $g = (1, 2, \dots, k)$ be a $k$-cycle, and let $1 \leq s, r < k$ with  $ \text{gcd}(s, r) = 1 $. Define $g^s = \tau_1  \cdots \tau_d$, $g^r = \sigma_1\cdots \sigma_m$, and $g^{s+r} = \gamma_1 \cdots \gamma_q$ as products of disjoint cycles, where $d = \gcd(s, k)$, $m = \gcd(r, k)$, and $q = \gcd(s+r, k)$. 
Then 
\begin{alignat}{1}
T_{g^{s}}^{0,\cdots,0}\boxtimes_{V^{\otimes k}}T_{g^{r}}^{0,\cdots,0}= & \sum_{t_{1},\cdots,t_{q}}\sum_{j}\frac{S_{t_{1}',j}\cdots S_{t_{q}',j}}{S_{0,j}^{k-d-m}}T_{g^{s+r}}^{t_{1},\cdots,t_{q}}.\label{0 and 000}
\end{alignat}

\end{theorem}
\begin{proof}
Recall from  Lemma \ref{identity condition }, we have
$$
T_{g^{s}}^{0,\cdots,0}\boxtimes_{V^{\otimes k}} T_{g^{r}}^{0,\cdots,0}=\sum_{t_{1},\cdots,t_{q}}n_{t_{1},\cdots,t_{q}}T_{g^{s+r}}^{t_{1},\cdots,t_{q}}.
$$ and $n_{t_{1},\cdots,t_{q}}$ satisfies \begin{equation}
\frac{1}{S_{0,j}^{k-d-m}}=\sum_{t_{1},\cdots,t_{q}}n_{t_{1},\cdots,t_{q}}S_{t_{1},j}\cdots S_{t_{q},j}.\label{G eq-1}
\end{equation}
It suffices to prove \begin{equation}
n_{t_{1},\cdots,t_{q}}=\sum_{j}\frac{S_{t_{1}',j}\cdots S_{t_{q}',j}}{S_{0,j}^{k-d-m}}\label{n_{t_1,...,t_q}}
\end{equation}
for $t_{1},\cdots,t_{q}\in\left\{ 0,1,\cdots,p\right\} .$

Now we divide the proof into three cases based on the value of $q$. 
\begin{itemize}

\item Case  $q=1$:  Substituting $q=1$ into (\ref{G eq-1}), we obtain $\frac{1}{S_{0,j}^{k-d-m}}=\sum_{t_{1}}n_{t_{1}}S_{t_{1},j}$.
It follows that $n_{t_{1}}=\sum_{j}\frac{S_{t_{1}',j}}{S_{j,0}^{k-d-m}},$
which matches (\ref{n_{t_1,...,t_q}}) for $q=1$. 

\item Case $q=2:$  We denote the multiplicity of the module \( T_{g^{s+r}}^{t_{1},t_{2}} \) in the fusion product \( T_{g^{s}}^{0,\cdots,0} \boxtimes T_{g^{r}}^{0,\cdots,0} \) by 
\[
n_{t_{1},t_{2}} = \left\langle T_{g^{s}}^{0,\cdots,0} \boxtimes T_{g^{r}}^{0,\cdots,0}, T_{g^{s+r}}^{t_{1},t_{2}} \right\rangle.
\]
Using (\ref{G eq-1}), the coefficients \( n_{t_{1},t_{2}} \) satisfy the equation 
\begin{equation}
\frac{1}{S_{0,j}^{k-d-m}} = \sum_{t_{1},t_{2}} n_{t_{1},t_{2}} S_{t_{1},j} S_{t_{2},j}.
\label{q=00003D2 identity}
\end{equation}
In the mean time, 
\begin{alignat}{1}
 & n_{t_{1},t_{2}}\nonumber \\
= & \left\langle T_{g^{s}}^{0,\cdots,0}\boxtimes T_{g^{r}}^{0,\cdots,0},\left(M^{t_{1}}\otimes V^{\otimes(k-1)}\right)\boxtimes\cdots\boxtimes\left(V\otimes M^{t_{2}}\otimes V^{\otimes(k-2)}\right)\boxtimes T_{g^{s+r}}^{0,\cdots,0}\right\rangle \nonumber \\
= & \left\langle \left(M^{t_{1}'}\otimes V^{\otimes(k-1)}\right)\boxtimes\cdots\boxtimes\left(V\otimes M^{t_{2}'}\otimes V^{\otimes(k-2)}\right)\boxtimes T_{g^{s}}^{0,\cdots,0}\boxtimes T_{g^{r}}^{0,\cdots,0},T_{g^{s+r}}^{0,\cdots,0}\right\rangle \nonumber \\
= & \left\langle \left(\left(M^{t_{1}'}\boxtimes M^{t_{2}'}\right)\otimes V^{\otimes\left(k-1\right)}\right)\boxtimes_{V^{\otimes k}}T_{g^{s}}^{0,\cdots,0}\boxtimes_{V^{\otimes k}}T_{g^{r}}^{0,\cdots,0},T_{g^{s+r}}^{0,\cdots,0}\right\rangle \nonumber \\
= & \left\langle \sum_{a}N_{t_{1}',t_{2}'}^{a'}M^{a',0,\cdots,0}\boxtimes_{V^{\otimes k}}T_{g^{s}}^{0,\cdots,0}\boxtimes_{V^{\otimes k}}T_{g^{r}}^{0,\cdots,0},T_{g^{s+r}}^{0,\cdots,0}\right\rangle \nonumber \\
= & \sum_{a}N_{t_{1}',t_{2}'}^{a'}\left\langle T_{g^{s}}^{0,\cdots,0}\boxtimes_{V^{\otimes k}}T_{g^{r}}^{0,\cdots,0},M^{a,0,\cdots,0}\boxtimes_{V^{\otimes k}}T_{g^{s+r}}^{0,\cdots,0}\right\rangle \nonumber \\
= & \sum_{a}N_{t_{1}',t_{2}'}^{a'}\left\langle T_{g^{s}}^{0,\cdots,0}\boxtimes_{V^{\otimes k}}T_{g^{r}}^{0,\cdots,0},T_{g^{s+r}}^{0,\cdots,0}\right\rangle \nonumber \\
= & \sum_{a}N_{t_{1}',t_{2}'}^{a'}n_{a,0}\nonumber \\
= & \sum_{a}N_{t_{2}',a}^{t_{1}}n_{a,0},\label{q=00003D2-1}
\end{alignat}
where in the second identity, we apply category theory, and in the third identity, we utilize Lemma \ref{move to 1st factor}.
Now the right-hand side of (\ref{q=00003D2 identity}) is 
\begin{equation}
\sum_{t_{1},t_{2}}n_{t_{1},t_{2}}S_{t_{1},j}S_{t_{2},j}=\sum_{a,t_{1},t_{2}}N_{t_{2}',a}^{t_{1}}S_{t_{1},j}S_{t_{2},j}n_{a,0}=\frac{S_{a,j}}{S_{0,j}}n_{a,0}.\label{q=00003D2(2)}
\end{equation}
 By (\ref{q=00003D2 identity}) and (\ref{q=00003D2(2)}), we have $\frac{1}{S_{0,j}^{k-d-m-1}}=S_{a,j}n_{a,0}.$
Thus, $\sum_{j}\frac{S_{a',j}}{S_{0,j}^{k-d-m-1}}=\sum_{j}S_{a',j}S_{a,j}n_{a,0}$. It follows that 
\[
n_{a,0}=\sum_{j}\frac{S_{a',j}}{S_{0,j}^{k-d-m-1}}.
\]

Combining with (\ref{q=00003D2-1}), we obtain 
\[
n_{t_{1},t_{2}}=\sum_{j}\left(\sum_{a}N_{t_{2}',a}^{t_{1}}S_{a',j}\right)\frac{1}{S_{0,j}^{k-d-m-1}}=\sum_{j}\left(\frac{S_{t_{1}',j}S_{t_{2}',j}}{S_{0,j}}\right)\frac{1}{S_{0,j}^{k-d-m-1}}
=\sum_{j}\frac{S_{t_{1}',j}S_{t_{2}',j}}{S_{0,j}^{k-d-m}},
\]
 where the second identity follows from  Lemma \ref{calculation of fusion and S-matrix} (3). 

\item Case $q\ge3$: Now we present the proof in three steps.

\emph{Step 1.} Prove that 

\begin{equation}
n_{t_{1},\cdots,t_{q}}=\sum_{s_{1},s_{2},\cdots,s_{q-1}}N_{t_{2}',s_{1}}^{t_{1}}N_{t_{3}',s_{2}}^{s_{1}}\cdots N_{t_{q}',s_{q-1}}^{s_{q-2}}n_{s_{q-1},0,\cdots,0}.\label{Step 1: n variables to one}
\end{equation}
Indeed, if we denote the multiplicity of the module $T_{g^{s+r}}^{t_{1},\cdots,t_{q}}$
in the fusion product $T_{g^{s}}^{0,\cdots,0}\boxtimes T_{g^{r}}^{0,\cdots,0}$
by $n_{t_{1},\cdots,t_{n}}=\left\langle T_{g^{s}}^{0,\cdots,0}\boxtimes T_{g^{r}}^{0,\cdots,0},T_{g^{s+r}}^{t_{1},\cdots,t_{q}}\right\rangle .$
Then 
\begin{alignat*}{1}
 & n_{t_{1},\cdots,t_{n}}\\
= & \left\langle T_{g^{s}}^{0,\cdots,0}\boxtimes T_{g^{r}}^{0,\cdots,0},\left(M^{t_{1}}\otimes V^{\otimes(k-1)}\right)\boxtimes\cdots\boxtimes\left(V^{\otimes(q-1)}\otimes M^{t_{q}}\otimes V^{\otimes(k-q)}\right)\boxtimes T_{g^{s+r}}^{0,\cdots,0}\right\rangle \\
= & \left\langle \left(M^{t_{1}'}\otimes V^{\otimes(k-1)}\right)\boxtimes\cdots\boxtimes\left(V^{\otimes(q-1)}\otimes M^{t_{q}'}\otimes V^{\otimes(k-q)}\right)\boxtimes T_{g^{s}}^{0,\cdots,0}\boxtimes T_{g^{r}}^{0,\cdots,0},T_{g^{s+r}}^{0,\cdots,0}\right\rangle \\
= & \left\langle \left(\left(M^{t_{1}'}\boxtimes_{V}\cdots\boxtimes_{V}M^{t_{q}'}\right)\otimes V^{\otimes\left(k-1\right)}\right)\boxtimes T_{g^{s}}^{0,\cdots,0}\boxtimes T_{g^{r}}^{0,\cdots,0},T_{g^{s+r}}^{0,\cdots,0}\right\rangle \\
= & \left\langle \sum_{s_{1},s_{2},\cdots,s_{q-1}}N_{t_{1}',t_{2}'}^{s_{1}'}N_{s_{1}',t_{3}'}^{s_{2}'}\cdots N_{s_{q-2}',t_{q}'}^{s_{q-1}'}M^{s_{q-1}',0,\cdots,0}\boxtimes T_{g^{s}}^{0,\cdots,0}\boxtimes T_{g^{r}}^{0,\cdots,0},T_{g^{s+r}}^{0,\cdots,0}\right\rangle \\
= & \sum_{s_{1},s_{2},\cdots,s_{q-1}}N_{t_{2}',s_{1}}^{t_{1}}N_{t_{3}',s_{2}}^{s_{1}}\cdots N_{t_{q}',s_{q-1}}^{s_{q-2}}\left\langle T_{g^{s}}^{0,\cdots,0}\boxtimes T_{g^{r}}^{0,\cdots,0},\left(M^{s_{q-1}}\otimes V^{\otimes\left(k-1\right)}\right)\otimes T_{g^{s+r}}^{0,\cdots,0}\right\rangle \\
= & \sum_{s_{1},s_{2},\cdots,s_{q-1}}N_{t_{2}',s_{1}}^{t_{1}}N_{t_{3}',s_{2}}^{s_{1}}\cdots N_{t_{q}',s_{q-1}}^{s_{q-2}}\left\langle T_{g^{s}}^{0,\cdots,0}\boxtimes T_{g^{r}}^{0,\cdots,0},T_{g^{s+1}}^{s_{q-1},0,\cdots,0}\right\rangle \\
= & \sum_{s_{1},s_{2},\cdots,s_{q-1}}N_{t_{2}',s_{1}}^{t_{1}}N_{t_{3}',s_{2}}^{s_{1}}\cdots N_{t_{q}',s_{q-1}}^{s_{q-2}}n_{s_{q-1},0,\cdots,0}
\end{alignat*}
where in the first, second, third, the last identities we use (\ref{g^(s+r)}), the properties of fusion category, Lemma \ref{move to 1st  factor}, and symmetry property of fusion rules in (\ref{Fusion rule property for untwsted modules}),  respectively.

\emph{Step 2. }Find $n_{s_{q-1},0,\cdots,0}.$

Now by using (\ref{Step 1: n variables to one}), we obtain 
\begin{alignat}{1}
  &\sum_{t_{1},\cdots,t_{q}}n_{t_{1},\cdots,t_{q}}S_{t_{1},j}\cdots S_{t_{q},j}\nonumber \\
= & \sum_{t_{1},\cdots,t_{q}}\sum_{s_{1},s_{2},\cdots,s_{q-1}}N_{t_{2}',s_{1}}^{t_{1}}N_{t_{3}',s_{2}}^{s_{1}}\cdots N_{t_{q}',s_{q-1},}^{s_{q-2}}S_{t_{1},j}\cdots S_{t_{q},j}n_{s_{q-1},0,\cdots,0}\nonumber \\
= & \sum_{s_{q-1}}\left(\sum_{t_{1},\cdots,t_{q},s_{1},s_{2},\cdots,s_{q-2}}N_{t_{2}',s_{1}}^{t_{1}}N_{t_{3}',s_{2}}^{s_{1}}\cdots N_{t_{q}',s_{q-1},}^{s_{q-2}}S_{t_{1},j}\cdots S_{t_{q},j}\right)n_{s_{q-1},0,\cdots,0}
\end{alignat}

Note that 
\begin{alignat*}{1}
 & \sum_{t_{1},\cdots,t_{q},s_{1},s_{2},\cdots,s_{q-2}}N_{t_{2}',s_{1}}^{t_{1}}N_{t_{3}',s_{2}}^{s_{1}}\cdots N_{t_{q}',s_{q-1},}^{s_{q-2}}S_{t_{1},j}\cdots S_{t_{q},j}\\
= & \sum_{t_{3},\cdots,t_{q},s_{1},s_{2},\cdots,s_{q-2}}N_{t_{3}',s_{2}}^{s_{1}}\cdots N_{t_{q}',s_{q-1},}^{s_{q-2}}\left(\sum_{t_{1},t_{2}}N_{t_{2}',s_{1}}^{t_{1}}S_{t_{1},j}S_{t_{2},j}\right)S_{t_{3},j}\cdots S_{t_{q},j}\\
= & \sum_{t_{3},\cdots,t_{q},s_{1},s_{2},\cdots,s_{q-2}}N_{t_{3}',s_{2}}^{s_{1}}\cdots N_{t_{q}',s_{q-1}}^{s_{q-2}}\left(\frac{S_{s_{1},j}}{S_{0,j}}\right)S_{t_{3},j}\cdots S_{t_{q},j}\\
= & \sum_{t_{4},\cdots,t_{q},s_{2},\cdots,s_{q-2}}N_{t_{4}',s_{3}}^{s_{2}}\cdots N_{t_{q}',s_{q-1}}^{s_{q-2}}\left(\sum_{t_{3,}s_{1}}N_{t_{3}',s_{2}}^{s_{1}}S_{s_{1},j}S_{t_{3},j}\right)\frac{1}{S_{0,j}}S_{t_{4},j}\cdots S_{t_{q},j}\\
= & \sum_{t_{4},\cdots,t_{q},s_{2},\cdots,s_{q-2}}N_{t_{4}',s_{3}}^{s_{2}}\cdots N_{t_{q}',s_{q-1}}^{s_{q-2}}\left(\frac{S_{s_{2},j}}{S_{0,j}}\right)\frac{1}{S_{0,j}}S_{t_{4},j}\cdots S_{t_{q},j}\\
= & \frac{S_{s_{q-1},j}}{S_{0,j}^{q-1}},
\end{alignat*}
 we obtain that 
\begin{equation}
\sum_{t_{1},\cdots,t_{q}}n_{t_{1},\cdots,t_{q}}S_{t_{1},j}\cdots S_{t_{q},j}=\sum_{s_{q-1}}\left(\frac{S_{s_{q-1},j}}{S_{0,j}^{q-1}}\right) n_{s_{q-1},0,\cdots,0}. \label{G-eq-3}
\end{equation}

Combining (\ref{G eq-1}) and (\ref{G-eq-3}),   we obtain 
\[
\frac{1}{S_{0,j}^{k-d-m}}=\sum_{s_{q-1}}\frac{S_{s_{q-1},j}}{S_{0,j}^{q-1}}n_{s_{q-1},0,\cdots,0}, j=0,1,\cdots,p.
\]
Consequently, we have
\[
\sum_{j}\frac{S_{s_{q-1}',j}}{S_{0,j}^{k-d-m-q+1}}=\sum_{j,s_{q-1}}S_{s_{q-1}',j}S_{s_{q-1},j}n_{s_{q-1},0,\cdots,0}.
\]
Thus, 
\begin{equation}
n_{s_{q-1},0,\cdots,0}=\sum_{j}\frac{S_{s_{q-1}',j}}{S_{0,j}^{k-d-m-q+1}}.\label{j_e-1,0,0,0}
\end{equation}

\emph{Step 3. }Find $n_{t_{1},\cdots,t_{q}}.$

Using (\ref{Step 1: n variables to one}) and (\ref{j_e-1,0,0,0}),
we have 
\begin{alignat*}{1}
 & n_{t_{1},\cdots,t_{q}}\\
= & \sum_{j}\left(\sum_{s_{1},s_{2},\cdots,s_{q-1}}N_{t_{1}',t_{2}'}^{s_{1}'}N_{s_{1}',t_{3}'}^{s_{2}'}\cdots N_{s_{q-2}',t_{q}'}^{s_{q-1}'}S_{s_{q-1}',j}\right)\frac{1}{S_{0,j}^{k-d-m-q+1}}\\
= & \sum_{j}\frac{S_{t_{1}',j}\cdots S_{t_{q}',j}}{S_{0,j}^{q-1}}\cdot\frac{1}{S_{0,j}^{k-d-m-q+1}}\\
= & \sum_{j}\frac{S_{t_{1}',j}\cdots S_{t_{p}',j}}{S_{0,j}^{k-d-m}},
\end{alignat*}
where in the second identity we use Lemma \ref{calculation of fusion and S-matrix}
(5). The proof is completed.

\end{itemize}
\end{proof}

\subsection{\label{subsec:fusion step-2}The fusion product of the $g^{s}$-twisted
module $T_{g^{s}}^{i_{1},\cdots,i_{d}}$ and the $g^{r}$-twisted module
$T_{g^{r}}^{j_{1},\cdots,j_{m}}$}

Let \( s, r, d, m \) be as defined earlier. In this subsection, we apply the formula (\ref{0 and 000}) from Subsection \ref{subsec: fusion step-1} to determine \( T_{g^{s}}^{i_1, \dots, i_d} \boxtimes_{V^{\otimes k}} T_{g^{r}}^{j_1, \dots, j_m} \).

By applying Theorem \ref{T_g(V) and T_g^s(0,..0)}, the following general formula can be easily derived:

\begin{theorem} \label{Fusion product general}  Let $g = \left(1, 2, \cdots, k\right)$ be a $k$-cycle, 
and let $1 \leq s, r < k$ with $\gcd(r, s) = 1$. Suppose 
$g^{s} = \tau_{1} \cdots \tau_{d}$, $g^{r} = \sigma_{1} \cdots \sigma_{m}$, 
and $g^{s+r} = \gamma_{1} \cdots \gamma_{q}$ are products of disjoint cycles, 
where $d = \gcd(s, k)$, $m = \gcd(r, k)$, and $q = \gcd(s + r, k)$. Then for any $i_{1},\cdots,i_{d},$
$j_{1},\cdots,j_{m},$ $t_{1},\cdots,t_{q}\in\left\{ 0,1,\cdots,p\right\} $,
\begin{alignat*}{1}
 & T_{g^{s}}^{i_{1},\cdots,i_{d}}\boxtimes_{V^{\otimes k}}T_{g^{r}}^{j_{1},\cdots,j_{m}}\\
= & \sum_{t_{1},\cdots,t_{q},j}\frac{S_{i_{1},j}\cdots S_{i_{d},j}S_{j_{1},j}\cdots S_{j_{m},j}S_{t_{1}',j}\cdots S_{t_{q}',j}}{S_{0,j}^{k}}T_{g^{s+r}}^{t_{1},\cdots,t_{q}}.
\end{alignat*}

\end{theorem}

\begin{proof}Similar to (\ref{g^(s+r)}), we have 
\[
T_{g^{s}}^{i_{1},\cdots,i_{d}}=\left(M^{i_{1}}\otimes V^{\otimes(k-1)}\right)\boxtimes\cdots\boxtimes\left(V^{\otimes(d-1)}\otimes M^{i_{d}}\otimes V^{\otimes(k-d)}\right)\boxtimes T_{g^{s}}^{0,\cdots,0}
\]
and 
\[
T_{g^{r}}^{j_{1},\cdots,j_{m}}=\left(M^{j_{1}}\otimes V^{\otimes(k-1)}\right)\boxtimes\cdots\boxtimes\left(V^{\otimes(m-1)}\otimes M^{j_{m}}\otimes V^{\otimes(k-m)}\right)\boxtimes T_{g^{r}}^{0,\cdots,0}.
\]
Then for $i_{1},\cdots,i_{d},j_{1},\cdots,j_{m}\in\left\{ 0,1,\cdots,p\right\} $,
we have 
\begin{alignat*}{1}
 & T_{g^{s}}^{i_{1},\cdots,i_{d}}\boxtimes_{V^{\otimes k}}T_{g^{r}}^{j_{1},\cdots,j_{m}}\\
= & \left(\left(M^{i_{1}}\otimes V^{\otimes(k-1)}\right)\boxtimes_{V^{\otimes k}}\cdots\boxtimes_{V^{\otimes k}}\left(V^{\otimes(d-1)}\otimes M^{i_{d}}\otimes V^{\otimes(k-d)}\right)\boxtimes_{V^{\otimes k}}T_{g^{s}}^{0,\cdots,0}\right)\\
 & \boxtimes_{V^{\otimes k}}\left(\left(M^{j_{1}}\otimes V^{\otimes(k-1)}\right)\boxtimes_{V^{\otimes k}}\cdots\boxtimes_{V^{\otimes k}}\left(V^{\otimes(m-1)}\otimes M^{j_{m}}\otimes V^{\otimes(k-m)}\right)\boxtimes_{V^{\otimes k}}T_{g^{r}}^{0,\cdots,0}\right)\\
= & \left(M^{i_{1}}\otimes V^{\otimes(k-1)}\right)\boxtimes_{V^{\otimes k}}\cdots\boxtimes_{V^{\otimes k}}\left(V^{\otimes(d-1)}\otimes M^{i_{d}}\otimes V^{\otimes(k-d)}\right)\\
 & \boxtimes\left(\left(M^{j_{1}}\otimes V^{\otimes(k-1)}\right)\boxtimes_{V^{\otimes k}}\cdots\boxtimes_{V^{\otimes k}}\left(V^{\otimes(m-1)}\otimes M^{j_{m}}\otimes V^{\otimes(k-m)}\right)\right)\boxtimes T_{g^{s}}^{0,\cdots,0}\boxtimes_{V^{\otimes k}}T_{g^{r}}^{0,\cdots,0}\\
= & \left(\left(M^{i_{1}}\boxtimes_{V}\cdots\boxtimes_{V}M^{i_{d}}\boxtimes_{V}M^{j_{1}}\boxtimes_{V}\cdots\boxtimes_{V}M^{j_{m}}\right)\otimes V^{\otimes\left(k-d-m\right)}\right)\boxtimes\left(T_{g^{s}}^{0,\cdots,0}\boxtimes_{V^{\otimes k}}T_{g^{r}}^{0,\cdots,0}\right),
\end{alignat*}
where in the last identity we use Lemma \ref{move to 1st  factor}.
By applying Lemma \ref{calculation of fusion and S-matrix} (4) and
(5) along with analogous computations from Theorem \ref{T_g(V) and T_g^s(0,..0)},
we can readily deduce that
\begin{alignat*}{1}
 & T_{g^{s}}^{i_{1},  \cdots,i_{d}}\boxtimes_{V^{\otimes k}}T_{g^{r}}^{j_{1},  \cdots,j_{m}}\\
= & \sum_{t_{1},\cdots,t_{q}}\left(\sum_{j}\frac{S_{i_{1},j} \cdots S_{i_{d},j}S_{j_{1},j}\cdots S_{j_{m},j}S_{t_{1}',j}\cdots S_{t_{q}',j}}{S_{0,j}^{k}}\right)T_{g^{s+r}}^{t_{1},\cdots,t_{q}}.
\end{alignat*}
 This completes our proof.

\end{proof}

\begin{remark} $\sum_{j}\frac{S_{i_{1},j}\cdots S_{i_{d},j}S_{j_{1},j}\cdots S_{j_{m},j}S_{t_{1}',j}\cdots S_{t_{q}',j}}{S_{0,j}^{k}}$
are non-negative integers.

\end{remark}

\section{The Fusion Product $\left(T_{g}\left(V\right)\right)^{\boxtimes k}$}

Let $g=\left(1, 2, \cdots,k\right)$ be a $k$-cycle. In this section,
we establish a formula of $\left(T_{g}^{0}\right)^{\boxtimes k}$.
For convenience, we will denote $M^{i_{1}}\otimes\cdots\otimes M^{i_{k}}$
by $M^{i_{1},\cdots,i_{k}}$. 

\begin{proposition} \label{twisted to k formula} Let $g=\left(1, 2, \cdots,k\right)$
be a $k$-cycle. Then
\[
\left(T_{g}^{0}\right)^{\boxtimes k}=\sum_{i_{1}, i_2, \cdots,i_{k}}\left(\sum_{a}S_{0,a}^{2-2h}\prod_{\ell=1}^{k}\frac{S_{i_{\ell}',a}}{S_{0,a}}\right)M^{i_{1}, i_2, \cdots,i_{k}},
\]
where $h=\frac{\left(k-1\right)\left(k-2\right)}{2}$. 
\end{proposition}
\begin{proof}
First note that 
\begin{equation}
\left(T_{g}^{0}\right)^{\boxtimes\left(k-1\right)}=\sum_{i}n_{i}T_{g^{k-1}}^{i}.\label{power k eq-1}
\end{equation}

Let $\lambda_{1}=\left(T_{g}^{0}\right)^{0}$, $\lambda_2=\lambda_{i,s}=\left(T_{g^{k-1}}^{i}\right)^{s}$ for
$s=0,...,k-1$ 
and $\mu=\left(M^{j,\cdots,j}\right)^{0}.$ Similar to arguments in Lemma \ref{identity condition }, one has  $a_{\lambda_{1}}=T_{g}^0$,
$a_{\lambda_{2}}=T_{g^{k-1}}^{i}$ and $a_{\mu}=M^{j,\cdots,j}.$
Note that
\begin{alignat*}{1}
 & \frac{S_{\lambda_{1},\mu}}{S_{0,\mu}}=\frac{S_{\left(T_{g}^{0}\right)^{0},\left(M^{j,\cdots,j}\right)^{0}}}{S_{\left(V^{\otimes k}\right)^{G},\left(M^{j,\cdots,j}\right)^{0}}}=\frac{\frac{1}{k}e^{\frac{2\pi in}{k}}S_{T_{g}^0,M^{j,\cdots,j}}}{\frac{1}{k}S_{V,M^{j}}^{k}}=\frac{\frac{1}{k}e^{\frac{2\pi in}{k}}S_{0,j}}{\frac{1}{k}S_{0,j}^{k}}=e^{\frac{2\pi in}{k}}S_{0,j}^{1-k}
\end{alignat*}
and 
\[
\frac{S_{\lambda_{2},\mu}}{S_{0,\mu}}=\frac{S_{\left(T_{g^{-1}}^{i}\right)^{s},\left(M^{j,\cdots,j}\right)^{0}}}{S_{\left(V^{\otimes k}\right)^{G},\left(M^{j,\cdots,j}\right)^{0}}}=\frac{\frac{1}{k}e^{\frac{2\pi in\left(k-1\right)}{k}}S_{T_{g^{-1}}^{i},M^{j,\cdots,j}}}{\frac{1}{k}S_{0,j}^{k}}=\frac{\frac{1}{k}e^{\frac{2\pi in\left(k-1\right)}{k}}S_{i,j}}{\frac{1}{k}S_{0,j}^{k}}=e^{\frac{2\pi in\left(k-1\right)}{k}}\cdot\frac{S_{i,j}}{S_{0,j}^{k}}.
\]
Let \( a_{\lambda_{1}}^{k-1} \) denote the fusion product of \( k-1 \) copies of \( a_{\lambda_{1}} \), and let \( \lambda_{1}^{k-1} \) denote the fusion product of \( k-1 \) copies of \( \lambda_{1} \).  Assume that $\lambda_1^{k-1}=\sum_{i,s}n_{\lambda_1}^{\lambda_{i,s}}\lambda_{i,s}.$  Then
\[
a_{\lambda_{1}}^{k-1}=a_{\lambda_{1}^{k-1}}=\sum_{i,s}n_{\lambda_{1}}^{\lambda_{i,s}}a_{\lambda_{i,s}},
\]
  That is,
 \[\left(T_{g}^{0}\right)^{\boxtimes\left(k-1\right)}=\sum_{i,s}n_{\lambda_{1}}^{\lambda_{i,s}} T_{g^{k-1}}^{i}.\] 
 From equation (\ref{power k eq-1})  we see that $n_i=\sum_{s=0}^{k-1}n_{\lambda_{1}}^{\lambda_{i,s}}.$
 By  \cite[Theorem 7.1]{DRX3} we have 
\[
\left(\frac{S_{\lambda_{1},\mu}}{S_{0,\mu}}\right)^{k-1}=\sum_{i,s}n_{\lambda_{1}}^{\lambda_{i,s}}\frac{S_{\lambda_{i,s},\mu}}{S_{0,\mu}},
\]
or equivalently,
\[
\left(e^{\frac{2\pi in}{k}}S_{0,j}^{1-k}\right)^{k-1}=\sum_{i,s}n_{\lambda_{1}}^{\lambda_{i,s}}e^{\frac{2\pi in\left(k-1\right)}{k}}\cdot\frac{S_{i,j}}{S_{0,j}^{k}}=\sum_{i}n_{i}e^{\frac{2\pi in\left(k-1\right)}{k}}\frac{S_{i,j}}{S_{0,j}^{k}}.
\]
That is, $\frac{1}{S_{0,j}^{k^{2}-3k+1}}=\sum_{i}n_{i}S_{i,j}$
for $j=0,1,\cdots,p$ and  $n_{i}=\sum_{j}\frac{S_{i',j}}{S_{0,j}^{k^{2}-3k+1}}.$
So we have
\begin{equation}
\left(T_{g}^{0}\right)^{\boxtimes\left(k-1\right)}=\sum_{i,j}\frac{S_{i',j}}{S_{0,j}^{k^{2}-3k+1}}\left(M^{i}\otimes V^{\otimes\left(k-1\right)}\right)\boxtimes_{V^{\otimes k}}T_{g^{k-1}}^{0}.\label{n_i}
\end{equation}

Since the fusion product of a $g$-twisted module and a $g^{k-1}$-twisted
module is an untwisted module. One can write 
\[
T_{g^{k-1}}^{0}\boxtimes_{V^{\otimes k}}T_{g}^{0}=\sum_{i_{1},  \cdots,i_{k}}m_{i_{1}, \cdots,i_{k}}M^{i_{1},  \cdots,i_{k}}.
\]
Now we have 
\begin{alignat}{1}
 & \left(T_{g}^{0}\right)^{\boxtimes k}\nonumber \\
= & \left(T_{g}^{0}\right)^{\boxtimes\left(k-1\right)}\boxtimes_{V^{\otimes k}}T_{g}^{0}\nonumber \\
= & \sum_{i,j}\frac{S_{i',j}}{S_{0,j}^{k^{2}-3k+1}}\left(M^{i}\otimes V^{\otimes\left(k-1\right)}\right)\boxtimes_{V^{\otimes k}}T_{g^{-1}}^{0}\boxtimes_{V^{\otimes k}}T_{g}^{0}\nonumber \\
= & \sum_{i,j}\frac{S_{i',j}}{S_{0,j}^{k^{2}-3k+1}}\left(M^{i}\otimes V^{\otimes\left(k-1\right)}\right)\boxtimes_{V^{\otimes k}}\left(\sum_{i_{1},\cdots,i_{k}}m_{i_{1},\cdots,i_{k}}M^{i_{1},\cdots,i_{k}}\right)\nonumber \\
= & \sum_{i,j,i_{1},\cdots,i_{k}}\frac{S_{i',j}}{S_{0,j}^{k^{2}-3k+1}}m_{i_{1},\cdots,i_{k}}\left(M^{i}\otimes V^{\otimes\left(k-1\right)}\right)\boxtimes_{V^{\otimes k}}M^{i_{1},\cdots,i_{k}}.\label{T_g(V)^k}
\end{alignat}
Note that by properties of fusion rules and dual module of twisted
modules in Propositions \ref{Fusion rule prop twisted} and \ref{twisted dual},
we have 
\[
m_{i_{1},\cdots,i_{k}}=N_{T_{g^{k-1}}^{0},T_{g}^{0}}^{M^{i_{1},\cdots,i_{k}}}=N_{T_{g}^{0}\circ g^{k-1},T_{g^{k-1}}^{0}}^{M^{i_{1},\cdots,i_{k}}}=N_{T_{g}^{0},M^{i_{1}',\cdots,i_{k}'}}^{\left(T_{g^{k-1}}^{0}\right)^{'}}=N_{T_{g}^{0},M^{i_{1}',\cdots,i_{k}'}}^{T_{g}^{0}}.
\]

By Theorem \ref{fusion product of untwisted and twisted-General},
\begin{align*}
T_{g}^{0}\boxtimes_{V^{\otimes k}}M^{i_{1}',\cdots,i_{k}'} & =T_{g}\left(M^{i_{1}'}\boxtimes_{V}\cdots\boxtimes_{V}M^{i_{k}'}\right) =\sum_{j_{1}',\cdots j_{k-1}'}N_{i_{1}',i_{2}'}^{j_{1}'}N_{j_{1}',i_{3}'}^{j_{2}'}\cdots N_{j_{k-3}',i_{k-1}'}^{j_{k-2}'}N_{j_{k-2}',i_{k}'}^{j_{k-1}'}T_{g}^{j_{k-1}'}.
\end{align*}
Thus 
\[
m_{i_{1},\cdots,i_{k}}=N_{T_{g}^{0},M^{i_{1}',\cdots,i_{k}'}}^{T_{g}^{0}}=\sum_{j_{1}',\cdots j_{k-2}'}N_{i_{1}',i_{2}'}^{j_{1}'}N_{j_{1}',i_{3}'}^{j_{2}'}\cdots N_{j_{k-3}',i_{k-1}'}^{j_{k-2}'}N_{j_{k-2}',i_{k}'}^{0}.
\]
Since $N_{j_{k-2}',i_{k}'}^{0}=N_{j_{k-2}',0}^{i_{k}},$ we obtain
$i_{k}=j_{k-2}'.$ Now using Lemma \ref{calculation of fusion and S-matrix} gives
\begin{alignat}{1}
 & m_{i_{1},\cdots,i_{k}}\nonumber \\
= & \sum_{j_{1}',\cdots,j_{k-3}'}N_{i_{1}',i_{2}'}^{j_{1}'}N_{j_{1}',i_{3}'}^{j_{2}'}\cdots N_{j_{k-5}',i_{k-3}'}^{j_{k-4}'}N_{j_{k-4}',i_{k-2}'}^{j_{k-3}'}N_{j_{k-3}',i_{k-1}'}^{i_{k}}N_{i_{k},i_{k}'}^{0}\nonumber \\
= & \sum_{j_{1}',\cdots,j_{k-3}'}N_{i_{1}',i_{2}'}^{j_{1}'}N_{j_{1}',i_{3}'}^{j_{2}'}\cdots N_{j_{k-5}',i_{k-3}'}^{j_{k-4}'}N_{j_{k-4}',i_{k-2}'}^{j_{k-3}'}\sum_{t}\frac{S_{j_{k-3}',t}S_{i_{k-1}',t}S_{i_{k}',t}}{S_{0,t}}\nonumber \\
= & \sum_{t}\left(\sum_{j_{1}',\cdots,j_{k-3}'}N_{i_{1}',i_{2}'}^{j_{1}'}N_{j_{1}',i_{3}'}^{j_{2}'}\cdots N_{j_{k-5}',i_{k-3}'}^{j_{k-4}'}N_{j_{k-4}',i_{k-2}'}^{j_{k-3}'}S_{j_{k-3}',t}\right)\cdot\frac{S_{i_{k-1}',t}S_{i_{k}',t}}{S_{0,t}}\nonumber \\
= & \sum_{t}\frac{S_{i_{1}',t}\cdots S_{i_{k}',t}}{S_{0,t}^{k-2}}\nonumber \\
= & \sum_{a}S_{0,a}^{2}\prod_{\ell=1}^{k}\frac{S_{i_{\ell}',a}}{S_{0,a}}.\label{m_(i_1,...,i_k)}
\end{alignat}

Plugging equation (\ref{m_(i_1,...,i_k)}) into (\ref{T_g(V)^k}), we
obtain 
\begin{align}
 & \left(T_{g}^{0}\right)^{\boxtimes k}\nonumber \\
= & \sum_{i,t,i_{1},\cdots,i_{k}}S_{0,t}^{-k^{2}+3k-1}S_{i',t}\cdot\sum_{a}S_{0,a}^{2}\prod_{\ell=1}^{k}\frac{S_{i_{\ell}',a}}{S_{0,a}}\cdot\left(M^{i}\otimes V^{\otimes\left(k-1\right)}\right)\boxtimes_{V^{\otimes k}}M^{i_{1}, i_2, \cdots,i_{k}}\nonumber \\
= & \sum_{i,t,i_{1},\cdots,i_{k},a}S_{0,t}^{-k^{2}+3k-1}S_{i',t}\cdot S_{0,a}^{2}\prod_{\ell=1}^{k}\frac{S_{i_{\ell}',a}}{S_{0,a}}\cdot\sum_{j}N_{i,i_{1}}^{j}M^{j,i_{2},\cdots,i_{k}}\nonumber \\
= & \sum_{i_{2},\cdots,i_{k},a,j}\left(\sum_{i,t,i_{1}}S_{0,t}^{-k^{2}+3k-1}S_{i',t}\cdot S_{0,a}^{2}\prod_{\ell=1}^{k}\frac{S_{i_{\ell}',a}}{S_{0,a}}\cdot N_{i,i_{1}}^{j}\right)M^{j,i_{2},\cdots,i_{k}}.\label{Section 5-1}
\end{align}

Note that
\begin{alignat}{1}
 & \sum_{i,t,i_{1}}S_{0,t}^{-k^{2}+3k-1}S_{i',t}\cdot S_{0,a}^{2}\prod_{\ell=1}^{k}\frac{S_{i_{\ell}',a}}{S_{0,a}}\cdot N_{i,i_{1}}^{j}\nonumber \\
= & \sum_{i,t}S_{0,t}^{-k^{2}+3k-1}S_{i',t}\cdot S_{0,a}\prod_{\ell=2}^{k}\frac{S_{i_{\ell}',a}}{S_{0,a}}\cdot\left(\sum_{i_{1}}S_{i_{1}',a}N_{i,j'}^{i_{1}'}\right)\nonumber \\
= & \sum_{i,t}S_{0,t}^{-k^{2}+3k-1}S_{i',t}\cdot S_{0,a}\prod_{\ell=2}^{k}\frac{S_{i_{\ell}',a}}{S_{0,a}}\cdot\left(\frac{S_{i,a}S_{j',a}}{S_{0,a}}\right)\nonumber \\
= & \sum_{i,t}S_{0,t}^{-k^{2}+3k}\left(S_{i',t}S_{i,a}\right)\cdot \prod_{\ell=2}^{k}\frac{S_{i_{\ell}',a}}{S_{0,a}}\cdot\frac{S_{j',a}}{S_{0,a}}\nonumber \\
= & S_{0,a}^{-k^{2}+3k}\prod_{\ell=2}^{k}\frac{S_{i_{\ell}',a}}{S_{0,a}}\cdot\frac{S_{j',a}}{S_{0,a}},\label{Section 5-2}
\end{alignat}
where in the second identity we use Lemma \ref{calculation of fusion and S-matrix}
and in the fourth identity we use Verlinde formula in Theorem \ref{Verlinde formula}.
Combining (\ref{Section 5-1}) and (\ref{Section 5-2}), we obtain
\begin{alignat*}{1}
\left(T_{g}^{0}\right)^{\boxtimes k} & =\sum_{i_{2},\cdots,i_{k},a,j}\left(S_{0,a}^{-k^{2}+3k}\prod_{\ell=2}^{k}\frac{S_{i_{\ell}',a}}{S_{0,a}}\cdot\frac{S_{j',a}}{S_{0,a}}\right)M^{j,i_{2},\cdots,i_{k}}\\
 & =\sum_{i_{1},i_{2},\cdots,i_{k},a}\left(S_{0,a}^{-k^{2}+3k}\prod_{\ell=1}^{k}\frac{S_{i_{\ell}',a}}{S_{0,a}}\right)M^{i_{1},i_{2},\cdots,i_{k}}\\
 & =\sum_{i_{1},i_{2},\cdots,i_{k}}\left(\sum_{a}S_{0,a}^{2-2h}\prod_{\ell=1}^{k}\frac{S_{i_{\ell}',a}}{S_{0,a}}\right)M^{i_{1},i_{2},\cdots,i_{k}},
\end{alignat*}
where $h=\frac{\left(k-1\right)\left(k-2\right)}{2}$.

\end{proof}

\begin{remark}
For an interesting perspective on the above fusion rules  from the construction of conformal blocks associated with Riemann surfaces, see \cite{G}.
\end{remark}

\section{Example: $\left(\left(V_{\mathbb{Z}\alpha}\right)^{\otimes4}\right)^{\left\langle g\right\rangle }$
with $g=\left(1\ 2\ 3\ 4\right)$}

In this section, we consider the permutation orbifold vertex operator
algebra $\left(\left(V_{\mathbb{Z}\alpha}\right)^{\otimes 4}\right)^{\left\langle g \right\rangle}$,
where $V_{\mathbb{Z}\alpha}$ is the rank one lattice vertex operator algebra
with $\left\langle \alpha, \alpha \right\rangle = 2n$ for $n > 0$. It is known that
all irreducible $V_{\mathbb{Z}\alpha}$-modules are $V_{\mathbb{Z}\alpha + \lambda_{\ell}}$, where
$\lambda_{\ell} = \frac{\ell}
{2n}$ with $\ell = 0, \dots, 2n-1$ \cite{D}.
Denote $M^{\ell} = V_{\mathbb{Z}\alpha + \lambda_{\ell}}$. Thus,
$\text{Irr}\left(V_{\mathbb{Z}\alpha}\right) = \left\{ M^{0}, M^{1}, \dots, M^{2n-1} \right\}$.
Referring to page 106 of \cite{S}, we have:
\[
Z_{\mathbb{Z}\alpha+\lambda_{\ell}}\left(-\frac{1}{\tau}\right)=\sum_{j=0}^{2n-1}\frac{1}{\sqrt{2n}}e^{-2\pi i\left(\lambda_{\ell},\lambda_{j}\right)}Z_{\mathbb{Z}\alpha+\lambda_{j}}\left(\tau\right).
\]
Thus the entries of $S$-matrix $\left(S_{\ell,j}\right)$ for $V_{\mathbb{Z}\alpha}$
is given by 
\begin{equation}
S_{\ell,j}:=S_{\lambda_{\ell},\lambda_{j}}=\frac{1}{\sqrt{2n}}e^{-2\pi i\left(\lambda_{\ell},\lambda_{j}\right)}=\frac{1}{\sqrt{2n}}e^{-2\pi i\left(\frac{\ell}{2n}\alpha,\frac{j}{2n}\alpha\right)}=\frac{1}{\sqrt{2n}}e^{-\frac{\pi\ell ji}{n}},\quad\ell,j=0,1,\cdots,2n-1.\label{S-matrix of lattice VOA of rank one}
\end{equation}

For $g=\left(1\ 2\ 3\ 4\right)$, it is easy to see that $g^{2}=\sigma_{1}\sigma_{2}$
where $\sigma_{1},\sigma_{2}$ are $2$-cycles. For $\ell, q_1, q_2\in\{0, 1, \cdots, 2n-1\}$, denote $T_{g}\left(M^{\ell}\right)$
and $T_{\sigma_{1}}\left(M^{q_{1}}\right)\otimes T_{\sigma_{2}}\left(M^{q_{2}}\right)$
by $T_{g}^{\ell}$ and $T_{g^{2}}^{q_{1},q_{2}}$ as before. Then by Theorem \ref{Fusion product general} and (\ref{S-matrix of lattice VOA of rank one}), we have the following fusion product of $g$-twisted $(V_{\mathbb Z\alpha})^{\otimes 4}$-modules:
\begin{alignat*}{1}
T_{g}^{\ell}\boxtimes_{V^{\otimes4}}T_{g}^{j} & =\sum_{q_{1}q_{2},t}\frac{S_{\ell,t}S_{j,t}S_{q_{1}',t}S_{q_{2}',t}}{S_{0,t}^{4}}T_{g^{2}}^{q_{1},q_{2}}\\
 & =\sum_{t=0}^{2n-1}\frac{\frac{1}{\sqrt{2n}}e^{-\frac{\pi\ell ti}{n}}\frac{1}{\sqrt{2n}}e^{-\frac{\pi jti}{n}}\frac{1}{\sqrt{2n}}e^{\frac{\pi q_{1}ti}{n}}\frac{1}{\sqrt{2n}}e^{\frac{\pi q_{2}ti}{n}}}{\left(\frac{1}{\sqrt{2n}}\right)^{4}}T_{g^{2}}^{q_{1},q_{2}}\\
 & =\sum_{t=0}^{2n-1}e^{-\frac{\pi(\ell+j-q_{1}-q_{2})ti}{n}}T_{g^{2}}^{q_{1},q_{2}},
\end{alignat*}
where $\ell, j, q_1, q_2\in\{0,1, \cdots, 2n-1\}.$
 
\subsection{Fusion products of twisted modules for $\left(V_{\mathbb{Z}\alpha}\right)^{\otimes4}$ with $\langle \alpha, \alpha\rangle=2$ and $g=(1\ 2\ 3\ 4)$}

If $n = 1$, then the irreducible modules of the  lattice vertex operator algebra $V_{\mathbb{Z}\alpha}$ are $M^0 = V_{\mathbb{Z}\alpha}$ and $M^1 = V_{\mathbb{Z}\alpha + \frac{1}{2}\alpha}$. First, we consider the fusion products of untwisted $V^{\otimes 4}$-modules with twisted $V^{\otimes 4}$-modules. Recall that any irreducible $V^{\otimes 4}$-module can be written as $ M^{i_1, i_2, i_3, i_4}:=M^{i_1} \otimes M^{i_2} \otimes M^{i_3} \otimes M^{i_4}$, where $i_1, \dots, i_4 \in \{0, 1\}$. By Theorem B in \cite{DLXY} and the fusion products for lattice vertex operator algebras in Proposition 12.9 of \cite{DLe}, we have:
\[
T_{g^j}^i \boxtimes_{V^{\otimes 4}} M^{i_1, i_2, i_3, i_4} = T_{g^j}^{\overline{i_1 + i_2 + i_3 + i_4 + i}}, i, i_1, i_2, i_3,i_4\in\{0, 1\}, j = 1, 3, \overline{n} \equiv n \pmod{2}.
\]

By Theorem C in \cite{DLXY} and the fusion products for lattice vertex operator algebras in Proposition 12.9 of \cite{DLe}, we also obtain
\[
T_{g^2}^{i, j} \boxtimes_{V^{\otimes 4}} M^{i_1, i_2, i_3, i_4} = T_{g^2}^{\overline{i_1 + i_2 + i}, \overline{i_3 + i_4 + j}}, i, j, i_1, i_2, i_3,i_4\in \{0, 1\}, \overline{n}\equiv n \pmod{2}.
\]

Now we consider the fusion products of twisted modules. We have
\[
T_g^\ell \boxtimes_{V^{\otimes 4}} T_g^j = \sum_{t=0}^{1} e^{-\pi(\ell + j - q_1 - q_2) t i} T_{g^2}^{q_1, q_2} = \left(1 + e^{-\pi(\ell + j - q_1 - q_2)}\right) T_{g^2}^{q_1, q_2},  \ell, j, q_1, q_2 \in \{0, 1\}.
\]
It is straightforward to derive the following fusion products for twisted $\left(V_{\mathbb{Z}\alpha}^{\otimes 4}\right)^{\langle g \rangle}$-modules:
\[
T_g^0 \boxtimes_{V^{\otimes 4}} T_g^0 = 2 T_{g^2}^{0, 0} + 2 T_{g^2}^{1, 1},
\]
\[
T_g^0 \boxtimes_{V^{\otimes 4}} T_g^1 = T_g^1 \boxtimes_{V^{\otimes 4}} T_g^0 = 2 T_{g^2}^{0, 1} + 2 T_{g^2}^{1, 0},
\]
\[
T_g^1 \boxtimes_{V^{\otimes 4}} T_g^1 = 2 T_{g^2}^{0, 0} + 2 T_{g^2}^{1, 1}.
\]

Let \(W\) be a twisted \(V^{\otimes 4}\)-module. Then, by \cite{DRX1}, the quantum dimension of \(W\) is
\[
\text{qdim}_{V^{\otimes 4}} W = \frac{S_{W, V^{\otimes 4}}}{S_{V^{\otimes 4}, V^{\otimes 4}}}=\frac{S_{W, V^{\otimes 4}}}{S_{V,V}^4}.
\]
By Lemma \ref{S-matrix, nontwisted and twisted} and (\ref{S-matrix of lattice VOA of rank one}), we find that the quantum dimensions of all irreducible twisted \(V^{\otimes k}\)-modules are as follows:

\begin{center}
\begin{tabular}{|c|c|c|c|c|c|c|c|c|}
\hline 
 & $T_g^0$ & $T_g^1$ & $T_{g^2}^{0, 0}$ & $T_{g^2}^{0, 1}$ & $T_{g^2}^{1, 0}$ & $T_{g^2}^{1, 1}$ & $T_{g^3}^0$ & $T_{g^3}^1$ \\
\hline 
$\text{qdim}$ & $2\sqrt{2}$ & $2\sqrt{2}$ & $2$ & $2$ & $2$ & $2$ & $2\sqrt{2}$ & $2\sqrt{2}$ \\
\hline 
\end{tabular}
\end{center}

Using the properties of fusion rules for twisted modules in Lemma \ref{Fusion rule prop twisted} and the properties of dual twisted modules in Proposition \ref{twisted dual}, we obtain \[
T_g^i \boxtimes_{V^{\otimes 4}} T_{g^2}^{j, \ell} = 2 T_{g^3}^{\overline{i + j + \ell}}, \quad i, j, \ell \in \{0, 1\}, \quad \overline{n} \equiv n \pmod{2}.
\]
By Theorem \ref{twisted to k formula}, we have
\begin{alignat*}{1}
\left(T_{g}^{0}\right)^{\boxtimes4} & =\sum_{i_{1},i_{2},i_3,i_{4}=0}^{1}\left(\sum_{a=0}^{1}S_{0,a}^{-4}\prod_{\ell=1}^{4}\frac{S_{i_{\ell}',a}}{S_{0,a}}\right)M^{i_{1},i_{2},i_{3},i_{4}}\\
 & =\sum_{i_{1},i_{2},i_3,i_{4}=0}^{1}4\left(1+\left(-1\right)^{i_{1}+i_{2}+i_{3}+i_{4}}\right)M^{i_{1},i_{2},i_{3},i_{4}}.
\end{alignat*}
Direct computation gives:
\begin{alignat*}{1}
\left(T_{g}^{0}\right)^{\boxtimes4} & =8M^{0,0,0,0}+8M^{1,1,0,0}+8M^{0,1,1,0}+8M^{0,0,1,1}\\
 & +8M^{1,0,1,0}+8M^{0,1,0,1}+8M^{1,0,0,1}+8M^{1,1,1,1}.
\end{alignat*}

\section*{Acknowledgment}
This work was partially conducted during the third author's visit to the University of California, Santa Cruz. The third author gratefully acknowledges the support and hospitality provided by the Department of Mathematics at the University of California, Santa Cruz.

\vskip10pt {\footnotesize{}{}{}{}{ }}\textbf{\footnotesize{}{}{}{}C.
Dong}{\footnotesize{}{}{}{}: Department of Mathematics, University
of California Santa Cruz, CA 95064 USA; }\texttt{\footnotesize{}{}{}dong@ucsc.edu}{\footnotesize\par}

\textbf{\footnotesize{}{}{}{}F. Xu}{\footnotesize{}{}{}{}: Department
of Mathematics, University of California, Riverside, CA 92521 USA;
}\texttt{\footnotesize{}{}{}xufeng@math.ucr.edu}{\footnotesize\par}

\textbf{\footnotesize{}{}{}{}N. Yu}{\footnotesize{}{}{}{}: School
of Mathematical Sciences, Xiamen University, Fujian, 361005, China;
}\texttt{\footnotesize{}{}{} ninayu@xmu.edu.cn}{\footnotesize\par}


\begin{thebibliography}{DLM1}

\bibitem[A1]{A1} T. Abe, $C_{2}$-cofiniteness of the $2$-cycle permutation orbifold models of minimal Virasoro vertex operator algebras, \emph{Commun. Math. Phys.} \textbf{303} (2011), 825--844.

\bibitem[A2]{A2} T. Abe, $C_{2}$-cofiniteness of 2-cyclic permutation orbifold models, \emph{Commun. Math. Phys.} \textbf{317} (2013), 425--445.

\bibitem[ABD]{ABD} T. Abe, G. Buhl, C. Dong, Rationality, Regularity, and $C_{2}$-cofiniteness, \emph{Trans. Amer. Math. Soc.} \textbf{356} (2004), 3391--3402.

\bibitem[BDM]{BDM} K. Barron, C. Dong, G. Mason, Twisted sectors for tensor product vertex operator algebras associated to permutation groups, \emph{Commun. Math. Phys.} \textbf{227} (2002), no. 2, 349--384.

\bibitem[BEKT]{BEKT}B. Bakalov, J. Elsinger, V.G. Kac,  I. Todorov, Orbifolds of lattice vertex algebras, \emph{Jpn. J. Math.}   18 (2023), no. 2, 169--274.

\bibitem[BHL]{BHL} K. Barron, Y.-Z. Huang, J. Lepowsky, An equivalence of two constructions of permutation-twisted modules for lattice vertex operator algebras, \emph{J. Pure Appl. Algebra} \textbf{210} (2007), no. 3, 797--826.


\bibitem[BK]{BK} B. Bakalov, A. Kirillov, Lectures on Tensor Categories and Modular Functors, University Lecture Series, \textbf{21}, 2000.

\bibitem[Bo]{Bo} R. E. Borcherds, Vertex algebras, Kac-Moody algebras, and the Monster, \emph{Proc. Natl. Acad. Sci. USA} \textbf{83} (1986), 3068--3071.

\bibitem[CKM]{CKM} T. Creutzig, S. Kanade, R. McRae, Tensor categories for vertex operator superalgebra extensions, Creutzig, Mem. Amer. Math. Soc.  \textbf{295} (2024), no. 1472, vi+181 pp.




\bibitem[CM]{CM} S. Carnahan, M. Miyamoto, Regularity of fixed-point vertex operator subalgebras, arXiv:1603.05645.

\bibitem[D]{D} C. Dong, Vertex algebras associated with even lattices, \emph{J. Algebra} \textbf{160} (1993), 245--265.

\bibitem[DLe]{DLe} C. Dong, J. Lepowsky, Generalized Vertex Algebras and Relative Vertex Operators, Progress in Math., Vol. 112, Birkhauser, Boston, 1993.

\bibitem[DL]{DL} C. Dong, X. Lin, Twisted Verlinde formula for vertex operator algebras, arXiv:2310.15563.

\bibitem[DLM1]{DLM1} C. Dong, H. Li, G. Mason, Simple currents and extensions of vertex operator algebras, \emph{Commun. Math. Phys.} \textbf{180} (1996), 671--707.

\bibitem[DLM2]{DLM2} C. Dong, H. Li, G. Mason, Twisted representations of vertex operator algebras, \emph{Math. Ann.} \textbf{310} (1998), 571--600.

\bibitem[DLM3]{DLM3} C. Dong, H. Li, G. Mason, Modular invariance of trace functions in orbifold theory and generalized moonshine, \emph{Commun. Math. Phys.} \textbf{214} (2000), 1--56.

\bibitem[DLN]{DLN} C. Dong, X. Lin, S. Ng, Congruence property in conformal field theory, \emph{Algebra Number Theory} \textbf{9} (2015), 2121--2166.

\bibitem[DLXY]{DLXY} C. Dong, H. Li, F. Xu, N. Yu, Fusion products of twisted modules in permutation orbifolds, \emph{Trans. Am. Math. Soc.} \textbf{377} (2024), 1717--1760.

\bibitem[DJX]{DJX} C. Dong, X. Jiao, F. Xu, Quantum dimensions and quantum Galois theory, \emph{Trans. Am. Math. Soc.} \textbf{365} (2013), 6441--6469.

\bibitem[DNR]{DNR} C. Dong, R. Ng, L. Ren, Orbifolds and minimal modular extensions, 	arXiv:2108.05225.

\bibitem[DR]{DR} C. Dong, L. Ren, Congruence property in orbifold theory, \emph{Proc. Amer. Math. Soc.} \textbf{146} (2018), 497--506.

\bibitem[DRX1]{DRX1} C. Dong, L. Ren, F. Xu, On orbifold theory, \emph{Adv. Math.} \textbf{321} (2017), 1--30.

\bibitem[DRX2]{DRX2} C. Dong, L. Ren, F. Xu, S-matrix in orbifold theory, \emph{J. Algebra} \textbf{568} (2021), 139--159.

\bibitem[DRX3]{DRX3} C. Dong, L. Ren, F. Xu, Coset constructions and Kac-Wakimoto Hypothesis, 	arXiv:2404.00778.

\bibitem[DXY1]{DXY1} C. Dong, F. Xu, N. Yu, 2-cyclic permutations of lattice vertex operator algebras, \emph{Proc. Amer. Math. Soc.} \textbf{144} (2016), 3207--3220.

\bibitem[DXY2]{DXY2} C. Dong, F. Xu, N. Yu, 2-permutations of lattice vertex operator algebras: higher rank, \emph{J. Algebra} \textbf{476} (2017), 1--25.

\bibitem[DXY3]{DXY3} C. Dong, F. Xu, N. Yu, The 3-permutation orbifold of a lattice vertex operator algebra, \emph{J. Pure Appl. Algebra} \textbf{222} (2018), no. 6, 1316--1336.

\bibitem[DXY4]{DXY4} C. Dong, F. Xu, N. Yu, S-matrix in permutation orbifolds, \emph{J. Algebra} \textbf{606} (2022), 851--876.

\bibitem[EGNO]{EGNO} P. Etingof, S. Gelaki, D. Nikshych, V. Ostrik, \emph{Tensor Categories}, \textbf{205} Mathematical Surveys and Monographs, American Mathematical Society, Providence, RI, 2015.

\bibitem[ENO]{ENO} P. Etingof, D. Nikshych, V. Ostrik, On fusion categories, \emph{Ann. of Math.} \textbf{162} (2005), 581--642.

\bibitem[FHL]{FHL} I. Frenkel, Y.-Z. Huang, J. Lepowsky, On axiomatic approaches to vertex operator algebras and modules, \emph{Mem. Amer. Math. Soc.} \textbf{104} (1993).

\bibitem[FLM]{FLM} I. B. Frenkel, J. Lepowsky, A. Meurman, Vertex Operator Algebras and the Monster, Pure and Applied Math. \textbf{134}, Academic Press, Massachusetts, 1988.
\bibitem[G]{G} B. Gui, Genus-zero permutation-twisted conformal
 blocks for tensor product vertex operator
 algebras: The tensor-factorizable case. arXiv:2111.04662.

\bibitem[H1]{H1} Y.-Z. Huang, A theory of tensor products for module categories for a vertex operator algebra, IV, \emph{J. Pure Appl. Algebra} \textbf{100} (1995), 173--216.

\bibitem[H2]{H2} Y.-Z. Huang, Vertex operator algebras and the Verlinde conjecture, \emph{Commun. Contemp. Math.} \textbf{10} (2008), no. 1, 103--154.

\bibitem[H3]{H3} Y.-Z. Huang, Rigidity and modularity of vertex tensor categories, \emph{Commun. Contemp. Math.} \textbf{10} (2008), suppl. 1, 871--911.

\bibitem[HL1]{HL1} Y.-Z. Huang, J. Lepowsky, A theory of tensor products for module categories for a vertex operator algebra, I, II, \emph{Selecta Mathematica}, New Series \textbf{1} (1995), 699--756; 757--786.

\bibitem[HL2]{HL2} Y.-Z. Huang, J. Lepowsky, A theory of tensor products for module categories for a vertex operator algebra, III, \emph{J. Pure Appl. Algebra} \textbf{100} (1995), 141--171.

\bibitem[K1]{K1} A. Kirillov Jr., Modular categories and orbifold models II, 2001, arXiv:math/0110221.

\bibitem[K2]{K2} A. Kirillov Jr., Modular categories and orbifold models, \emph{Commun. Math. Phys.} \textbf{229} (2002), 309--335.

\bibitem[KL]{KL} M. Karel, H. Li, Certain generating subspaces for vertex operator algebras, \emph{J. Algebra} \textbf{217} (1999), 393--421.

\bibitem[KLX]{KLX} V. Kac, R. Longo, F. Xu, Solitons in affine and permutation orbifolds, \emph{Commun. Math. Phys.} \textbf{253} (2005), no. 3, 723--764.

\bibitem[KO]{KO} A. Kirillov Jr., V. Ostrik, On a $q$-analogue of the McKay correspondence and the ADE classification of $sl_2$ conformal field theories, \emph{Adv. Math.} \textbf{171} (2002), 183--227.

\bibitem[Le]{Le} J. Lepowsky, Calculus of twisted vertex operators, \emph{Proc. Natl. Acad. Sci. USA} \textbf{82} (1985), 8295--8299.

\bibitem[Li1]{Li1} H. Li, An analogue of the Hom functor and a generalized nuclear democracy theorem, \emph{Duke Math. J.} \textbf{93} (1998), no. 1, 73--114.

\bibitem[Li2]{Li2} H. Li, Some finiteness properties of regular vertex operator algebras, \emph{J. Algebra} \textbf{212} (1999), 495--514.

\bibitem[Lu]{Lu} G. Lusztig, Leading coefficients of character values of Hecke algebras, \emph{Proc. Symp. Pure Math.} \textbf{47} (1987), 235--262.

\bibitem[M1]{M1} M. Miyamoto, A $\mathbb{Z}_3$-orbifold theory of lattice vertex operator algebra and $\mathbb{Z}_3$-orbifold constructions, \emph{Springer Proc. Math. Stat.} \textbf{40}, 319--344, 2013.

\bibitem[M2]{M2} M. Miyamoto, $C_2$-cofiniteness of cyclic-orbifold models, \emph{Commun. Math. Phys.} \textbf{335} (2015), 1279--1286.
\bibitem[MPS1]{MPS1} A. Milas, M. Penn,  C. Sadowski, Permutation orbifolds of Virasoro vertex algebras and W-algebras, \emph{J. Algebra}    \textbf{570} (2021), 267--296.
\bibitem[MPS2]{MPS2}A.  Milas,  M. Penn,  C. Sadowski, $S_3$-permutation orbifolds of Virasoro vertex algebras, \emph{J. Pure Appl. Algebra}   \textbf{ 227} (2023), no. 10, Paper No. 107378, 25 pp.
\bibitem[MPSh]{MPSh} A. Milas, M. Penn,  H. Shao,  Permutation orbifolds of the Heisenberg vertex algebra  $H(3)$, \emph{J. Math. Phys.} \textbf{ 60} (2019), no. 2, 021703, 17 pp.


\bibitem[Mu]{Mu} M. Mueger, Modular Categories, arXiv:1201.6593.

\bibitem[S]{S} J.-P. Serre, A Course in Arithmetic, Grad. Texts in Math., vol. 7, Springer-Verlag, 1973.

\bibitem[V]{V} E. Verlinde, Fusion rules and modular transformation in 2D conformal field theory, \emph{Nucl. Phys. B} \textbf{300} (1988), 360--376.

\bibitem[X1]{X1} F. Xu, Some computations in the cyclic permutations of completely rational nets, \emph{Commun. Math. Phys.} \textbf{267} (2006), no. 3, 757--782.

\bibitem[X2]{X2} X. Xu, Intertwining operators for twisted modules of a colored vertex operator superalgebra, \emph{J. Algebra} \textbf{175} (1995), no. 1, 241--273.

\bibitem[Z]{Z} Y. Zhu, Modular invariance of characters of vertex operator algebras, \emph{J. Amer. Math. Soc.} \textbf{9} (1996), 237--302.

\end{thebibliography}
\end{document}